\documentclass[11pt]{amsart}
\usepackage{amssymb,amsthm,amsmath,amstext}
\usepackage{graphicx}
\usepackage{mathrsfs}
\usepackage{bm}
\usepackage{multirow}
\usepackage[all]{xy}
\usepackage{tikz-cd}
\usepackage{xfrac}
\usepackage{anyfontsize}

\usepackage[left=1.25in,top=1.35in,right=1.25in,bottom=1.35in]{geometry}

\usepackage{color}

\theoremstyle{plain}
\newtheorem{thm}{Theorem}\newtheorem*{thm*}{Theorem}
\newtheorem*{conj*}{Conjecture}

\newtheorem{pr}[thm]{Proposition}
\newtheorem{lem}[thm]{Lemma}
\newtheorem{cor}[thm]{Corollary}

\theoremstyle{remark}

\theoremstyle{definition}
\newtheorem{defn}[thm]{Definition}

\theoremstyle{remark}
\newtheorem{rem}[thm]{Remark}

\newcommand{\ZZ}{\mathbb{Z}}
\newcommand{\QQ}{\mathbb{Q}}
\newcommand{\RR}{\mathbb{R}}
\newcommand{\CC}{\mathbb{C}}

\newcommand{\PP}{\mathbb{P}}

\newcommand{\sheaf}[1]{\mathcal{#1}}

\newcommand{\ko}{\sheaf{O}}

\usepackage[backref=page]{hyperref}
\hypersetup{pdftitle={}}
\hypersetup{pdfauthor={}}
\hypersetup{colorlinks=true,linkcolor=blue,anchorcolor=blue,citecolor=blue}

\begin{document}

\title[The cubo-cubic transformation of the smooth quadric 4-fold is very special]{The cubo-cubic transformation of the smooth quadric fourfold is very special}
\author{Jordi Hern\'andez}
\address{Institut de Math\'ematiques de Toulouse ; UMR 5219 \\ 
UPS, F-31062 Toulouse Cedex 9, France}
\email{jordi\_emanuel.hernandez\_gomez@math.univ-toulouse.fr}
\begin{abstract}
We classify special self-birational transformations of the smooth quadric threefold and fourfold, $Q^3$ and $Q^4$. It turns out that there is only one such example in each dimension. In the case of $Q^3$, it is given by the linear system of quadrics passing through a rational normal quartic curve. In the case of $Q^4$, it is given by the linear system of cubics passing through a non-minimal K3 surface of degree $10$ with two skew $(-1)$-lines. 
\end{abstract}
\maketitle
\section{Introduction}

A Cremona transformation $\PP^k\dashrightarrow \PP^k$ with a smooth and irreducible base locus scheme is said to be special. These maps stand out due to their simple structure, since the blow-up along the locus where they are not defined induces a resolution of the map by an everywhere defined birational morphism which is easier to understand. This simple structure has allowed to classify many cases by fixing the dimension of the base locus scheme or the type of the rational map. For example, the classification of special Cremona transformations with a base locus scheme of dimension at most $2$ was obtained by B. Crauder and S. Katz in \cite{crauderkatz} following this approach. Soon after this success, interest on the classification problem started to grow and new results were found in \cite{einshepherdbarron}, \cite{crauder.katz-harts}, and \cite{hulek.katz.Schreyer}. 

Later on, the classification problem was generalized to consider special birational transformations $ \PP^k \dashrightarrow Z$, where $Z$ is a sufficiently nice and well-behaved variety. In this vein, A. Alzati and J. C. Sierra generalized previous results to the case where $Z$ is a prime Fano manifold in \cite{alzati.sierra}, and G. Staglian\`o classified special quadratic birational maps into locally factorial varieties $Z$ in \cite{stagliano.specialquadhypersurface}, \cite{stagliano.specialquadbaselocusatmost3}, and \cite{stagliano.specialquadintersectionofquadrics}. Only recently, G. Staglian\`o completed the classification of special Cremona transformations with a base locus scheme of dimension $3$ in \cite{stagliano.specialcubicP6P7} and, more generally, the classification of special birational maps into factorial complete intersection varieties $Z$ with the same dimension condition in \cite{stagliano.specialcubicbir}. 

In any case, only special birational transformations with domain $\PP^k$ have been considered. The next problem to consider is to classify special birational transformations whose domain is a rational variety other than projective space. In fact, some potential examples of special birational transformations between rational varieties can be found in recent works related to the classification problem of Fano fourfolds. For example, in the paper \cite{Fano4foldsK3type} by M. Bernardara, E. Fatighenti, L. Manivel, and F. Tanturri, a list of $64$ new families of Fano fourfolds of K3 type is provided and in many examples the fourfolds from these families admit multiple birational contractions that can be realized as blow-ups of
Fano manifolds along non-minimal K3 surfaces. A particular example, labeled as K3-33 in their list, stands out for its appealing symmetry. It can be described as the blow-up of the smooth quadric hypersurface $Q^4\subseteq \PP^5$ along a non-minimal K3 surface of degree $10$ with two skew $(-1)$-lines in two different ways, and the rational map it induces is a special cubo-cubic birational transformation of $Q^4$. According to our terminology, based on that used by M. Gross in his classification of smooth surfaces of degree $10$ in the four-dimensional smooth quadric \cite{gross.degree10}, these surfaces are of type $^{II} Z_{E}^{10}$. Moreover, it was shown that the minimal models of the surfaces are derived equivalent K3 surfaces. Nevertheless, in the present paper we prove that they are non-isomorphic in general. 

This particular example leads naturally to the problem of classifying special self-birational transformations of $k$-dimensional smooth quadrics $Q^k$. In this paper we classify special self-birational transformations of $Q^3$ and $Q^4$. For the three-dimensional case we have the following theorem that answers a question posed by F. Russo and G. Staglian\`o after the first version of this paper appeared.

\begin{thm}\label{classificationSpecialtransQ3}
    Let $\varphi:Q^3\dashrightarrow Q^3$ be a special birational transformation. Then $\varphi$ is a quadro-quadric transformation, and its base locus scheme $C$ is a rational normal quartic curve. 
    
    Conversely, let $C\subseteq Q^3$ be any rational normal quartic curve. Then the linear system of quadrics in $Q^3$ passing through $C$ induces a special birational transformation $\varphi:Q^3\dashrightarrow Q^3$ with $C$ as base locus scheme. Moreover, the base locus scheme of $\varphi ^{-1}$ is a rational normal quartic curve as well.
\end{thm}

Theorem \ref{classificationSpecialtransQ3} follows easily from arithmetic considerations. We can think of this special transformation as the quadric analogue of the special cubo-cubic transformation of $\PP^3$, classified by S. Katz in his article ``The cubo-cubic transformation of $\PP^3$ is very special'' \cite{katz33}. The four-dimensional case is much more interesting. The main result of this paper is the following theorem.

\begin{thm} \label{main thm}
Let $\varphi:Q^4 \dashrightarrow Q^4$ be a special birational map. Then $\varphi$ is a cubo-cubic transformation, and its base locus scheme is a surface $S$ of type $^{II} Z_{E}^{10}$.

Conversely, let $S\subseteq Q^4$ be a surface of type $^{II} Z_{E}^{10}$ and let $S_0$ be its minimal model. Then the linear system $|H_S+K_S|$ induces a morphism $S\rightarrow \PP^7$ realizing the blow-up of $S_0$ at two points $p$ and $p'$. If $H_{S_0}$ is the hyperplane class of $S_0\subseteq \PP^7$ and there is no smooth elliptic curve $\mathfrak{E}\subseteq S_0$ of degree $5$ such that $\left\langle H_{S_0}, \mathfrak{E}\right\rangle\subseteq Pic(S_0)$ is saturated and the points $p$ and $p'$ belong to a (possibly singular) fiber of the associated elliptic fibration of $S_0$, then the linear system of cubics passing through $S$ induces a special birational transformation $\varphi:Q^4 \dashrightarrow Q^4$ with $S$ as base locus scheme. If moreover $rkPic(S_0)=1$, then the base locus scheme $T$ of $\varphi^{-1}$ is a smooth surface of type $^{II} Z_{E}^{10}$ as well. In fact, the K3 surfaces $S_0$ and $T_0$ are non-isomorphic Fourier-Mukai partners. More precisely, $T_{0}$ is the moduli space of stable sheaves of $S_0$ with Mukai vector $(2,H_{S_0},3)$.
\end{thm}

There are many points of comparison between this special self-birational transformation of $Q^4$ and a Cremona transformation of $\PP^4$ studied by B. Hassett and K. W. Lai in \cite{hassett.lai}. In their paper they classified Cremona transformations of $\PP^4$ whose base locus scheme is given by a surface with only transverse double points as singularities and they showed that there is just one such example. We can see the special cubo-cubic transformation of $Q^4$ as the smooth version of this Cremona transformation. 

The paper is organized as follows. In section \ref{section smooth surfaces in Q4}, we review basic results from the theory of congruences of lines in $\PP^3$. We also include results from \cite{gross.degree10}, which later will allow us to identify the base locus scheme $S$ of a special map $Q^4\dashrightarrow Q^4$ as a surface of type $^{II} Z_{E}^{10}$. In section \ref{section K3-33}, we present the construction of the family of Fano fourfolds of K3 type known as K3-33 given in \cite{Fano4foldsK3type}, and we show that the minimal models of the degree $10$ surfaces in $Q^4$ that are blown up are non-isomorphic K3 surfaces if their Picard rank is $1$. We follow the lattice-theoretic approach of B. Hassett and K. W. Lai from \cite{hassett.lai}. In section \ref{section special bir trans}, we quickly recall some of the definitions and properties of special rational maps, and then we focus on the proofs of Theorem \ref{classificationSpecialtransQ3} and Theorem \ref{main thm}. More precisely, we assume first that special self-birational maps of $Q^3$ and $Q^4$ exist and derive arithmetic conditions that can only be realized by the maps described in the theorems above. Then we make the classification effective by showing how to construct them in the most general setting.

\textit{Acknowledgements.} I would like to thank Francesco Russo and Giovanni Staglian\`o for many insightful comments and for asking me about the classification in dimension three. I would also like to thank Laurent Manivel for reading the preliminary version of this paper and asking me some questions that are now part of the present work. The content of this paper is part of my Ph.D. thesis.

\section{Smooth surfaces in \texorpdfstring{$Q^4$}{TEXT}}\label{section smooth surfaces in Q4}

\subsection{Invariants of a congruence.}

We review some key results from the theory of congruences of lines in $\PP^3$, which classically refers to the systematic study of smooth surfaces in $Q^4 \cong G(2,4)=\mathbb{G}(1,\PP ^3)$. For a good introduction to the subject, see \cite{arrondosols92}.

Let $S\subseteq \PP^5$ be a smooth surface of degree $\mathfrak{d} $ contained in $Q^4$. We also say that $S$ is a congruence of lines of degree $\mathfrak{d}$. In the Chow ring of $Q^4 $ we can write $S=\mathfrak{a}P_1 + \mathfrak{b}P_2 \in CH^2(Q^4)=\ZZ P_1 \oplus\ZZ P_2$, where $P_1$ and $P_2$ are the classes of planes from each of the rulings of $Q^4$ and $\mathfrak{a}$ and $\mathfrak{b}$ are integers. The pair $(\mathfrak{a},\mathfrak{b})$ is called the bidegree of $S$. Under the isomorphism $Q^4\cong \mathbb{G}(1,\PP ^3)$, the integer $\mathfrak{a}$ corresponds to the number of lines of the congruence passing through a fixed general point in $\PP^3$, while the integer $\mathfrak{b}$ corresponds to the number of lines of the congruence contained in a fixed general plane in $\PP^3$. Clearly, $\mathfrak{d}=\mathfrak{a}+\mathfrak{b}$ and $\mathfrak{a},\mathfrak{b}\geq 0$.

We denote by $H$ the hyperplane class of $Q^4\subseteq \PP^5$, by $H_S$ its restriction to $S$, by $K_S$ the canonical divisor of $S$, and by $\pi$ the sectional genus of $S$.

A complex is an effective Cartier divisor in $Q^4$. If $0\neq s\in H^0(Q^4, \ko(k))$ and $k\geq 1$, then the zero locus $V(s)\subseteq Q^4$ of the section $s$ is a complex whose degree is defined to be $k$. Clearly, all complexes are obtained in this way. If $k=1$, $2$ or $3$, we say that $V(s)$ is a linear, a quadratic, or a cubic complex, respectively. 

We collect a few formulas computing important invariants of the congruence $S$. 

\begin{pr}\label{total chern class of normal bundle of S}
    The total Chern class of $N_{S/Q^4}$ is given by 
    \[c(N_{S/Q^4})=1+ (K_S +4H_S) t + (7\mathfrak{d} +4H_S K_S -c_2(S) + K_S ^2) t^2.\]
\end{pr}
\begin{proof}
    See the proof of \cite[Proposition~2.1]{arrondosols92}.
\end{proof}

\begin{cor}\label{euler characteristic of normal bundle of S}
    The Euler characteristic of $N_{S/Q^4}$ is given by
    \[\chi(N_{S/Q^4})= 6\mathfrak{d} -\mathfrak{a}^2-\mathfrak{b}^2 +2(2\pi-2)+2\chi (\ko_S).\]
\end{cor}
\begin{proof}
    Let $\mathbb{V}$ be a vector bundle over $S$. Then a Grothendieck-Riemann-Roch computation gives the following formula 
    \[ \chi(\mathbb{V})=rk(\mathbb{V})\chi(\ko_S) +\frac{c_1(\mathbb{V})^2 -c_1(\mathbb{V})K_S}{2}-c_2(\mathbb{V}).\]
    The claim follows by taking $\mathbb{V}=N_{S/Q^4}$ and using Proposition \ref{total chern class of normal bundle of S}.
\end{proof}
    
\begin{cor}\label{square of K_S} The self-intersection number of $K_S$ is given by 
\[K_S ^2 = c_2(S)-4H_S K_S -7\mathfrak{d} + \mathfrak{d} ^2 -2\mathfrak{a}\mathfrak{b}.\]
\end{cor}
\begin{proof}
    The second Chern class of $N_{S/Q^4}$ is identified with the intersection product $S^2$. Since $S=\mathfrak{a}P_1+\mathfrak{b}P_2$ in $CH^2(Q^4)$, we have $c_2(N_{S/Q^4})=S^2=(\mathfrak{a}P_1+\mathfrak{b}P_2)^2=\mathfrak{a}^2 + \mathfrak{b}^2=\mathfrak{d} ^2 -2\mathfrak{a}\mathfrak{b}$. The claim follows by comparing this equality with the equality $c_2(N_{S/Q^4})=7\mathfrak{d} +4H_S K_S -c_2(S) + K_S ^2$ according to Proposition \ref{total chern class of normal bundle of S}.
\end{proof}

\subsection{Sectional genus of a congruence.}

Define the maximal sectional genus $\pi_ 1(\mathfrak{d})$ of the surface $S$ as the integer
\begin{equation*}
  \pi_ 1(\mathfrak{d}) : =
    \begin{cases}
      \frac{\mathfrak{d}^2 -4\mathfrak{d} +8}{8} & \text{if }\mathfrak{d}=0 \text{ mod } 4 ; \\
      \frac{\mathfrak{d}^2 -4\mathfrak{d} +3}{8} & \text{if }\mathfrak{d}=1,3 \text{ mod } 4 ; \\
      \frac{\mathfrak{d}^2 -4\mathfrak{d} +4}{8} & \text{if }\mathfrak{d}=2 \text{ mod } 4.
    \end{cases}       
\end{equation*}

The next theorem explains the choice of this name. 

\begin{thm}\label{maximalsectionalgenus} If $S$ is non-degenerate and $\mathfrak{d}\geq 9$, then $\pi \leq \pi_ 1(\mathfrak{d})$. 
\end{thm}
\begin{proof}
    See \cite[Theorem~3.6]{gross.distribution}.
\end{proof} 

The following result gives a lower bound for the arithmetic genus $p_ a$ of $S$ in terms of its degree $\mathfrak{d}$, as well as the existence of certain complexes containing $S$ when the maximal sectional genus is attained. 

\begin{thm}\label{lowerboundofp.a} If $S$ is non-degenerate, $\mathfrak{d}\geq 9$, and $\pi = \pi_ 1(\mathfrak{d})$, then 
\begin{equation*}
  p_ a \geq
    \begin{cases}
      \frac{\mathfrak{d}^3 -12\mathfrak{d}^2 + 56\mathfrak{d} -96}{96} & \text{if }\mathfrak{d}=0 \text{ mod } 4;\\
      \frac{\mathfrak{d}^3 -12\mathfrak{d}^2 + 41\mathfrak{d} -30}{96} & \text{if }\mathfrak{d}=1 \text{ mod } 4;\\
      \frac{\mathfrak{d}^3 -12\mathfrak{d}^2 + 44\mathfrak{d} -48}{96} & \text{if }\mathfrak{d}=2 \text{ mod } 4;\\
      \frac{\mathfrak{d}^3 -12\mathfrak{d}^2 + 41\mathfrak{d} -42}{96} & \text{if }\mathfrak{d}=3 \text{ mod } 4
    \end{cases}       
\end{equation*}
and $S$ is contained in a quadratic complex. Furthermore, equality holds if and only if $S$ is also contained in an irreducible complex of degree $\lceil \frac{\mathfrak{d}}{4}\rceil$. 
\end{thm}
\begin{proof}
    See \cite[Theorem~2.6]{gross.degree10}.
\end{proof} 

\subsection{Congruences of degree 10}

We include here the following classification theorem by M. Gross. We denote by ${\PP^2}^{(n)}$ the blow-up of $n$ points in $\PP^2$.

\begin{thm}\label{classif deg 10 surfaces thm}
Let $S\subseteq Q^4$ be a non-degenerate smooth surface of degree $10$. Then either the bidegree of $S$ is $(3,7)$, and it is a minimal Enriques surface, or it is of bidegree $(4,6)$, and it is one of the following types:
\begin{itemize}
    \item ${\PP^2} ^{(12)}$, embedded via $|7L-2E_1 -\cdots -2E_9 -E_{10} - E_{11} - E_{12} |$. ($Z_A ^{10}$)
    \item A K3 surface blown up in one point. ($Z_B ^{10}$)
\end{itemize}
Or it is of bidegree $(5,5)$, and it is one of the following types:
\begin{itemize}
    \item A ruled surface over an elliptic curve with fibers embedded with degree $2$, with $\pi=4,5$, or $6$. ($C_A ^{10}$, $C_B ^{10}$, and $C_C ^{10}$)
    \item ${ \PP^2 }^{(13)}$, embedded via $|7L-3E_1 -2 E_2-\cdots -2E_7 -E_8 -\cdots - E_{11} - E_{13} |$. ($Z_C ^{10}$)
    \item ${\PP^2} ^{(17)}$, embedded via $|6L-2E_1 -2E_2 -2E_3 -E_4 -\cdots  - E_{17} |$. ($Z_D ^{10}$)
    \item A K3 surface blown up in two points. ($Z_E ^{10}$)
    \item A minimal elliptic surface with $p_a = p_g =2$ and $\pi = 8$. ($Z_F ^{10}$)
\end{itemize}
\end{thm}
\begin{proof}
    See \cite[Theorem~1]{gross.degree10}.
\end{proof}

Our work is mainly concerned with surfaces $S$ of type $Z_E ^{10}$. These surfaces contain two skew $(-1)$-lines whose contraction defines a K3 surface of genus $7$. We can also describe their ideal sheaves $\mathcal{I}_{S/Q^4}$ via a resolution of known vector bundles, such as the so-called spinor bundles $\mathcal{E}$ and $\mathcal{E}'$ of $Q^4$, and the vector bundle $\dot{\mathcal{E}}$ (see \cite[\S 1.3]{arrondosols92} for its construction). Under the isomorphism $Q^4\cong G(2,4)$, the spinor bundles $\mathcal{E}$ and $\mathcal{E}'$ correspond to the tautological bundle $\mathcal{U}_{G(2,4)}$ and the dual of the tautological quotient bundle $\mathcal{Q}_{G(2,4)}$, respectively.

\begin{pr}\label{resolution of ideal sheaf of S types E and B}
Let $S\subseteq Q^4$ be a surface of type $Z_B ^{10}$. Then $S$ is not contained in a quadratic complex and has ideal sheaf resolution 
\[0\rightarrow \mathcal{E}(-3)\oplus \mathcal{E}(-3)\oplus \ko_{Q^4}(-4)\rightarrow \ko_{Q^4}(-3)^{\oplus 6} \rightarrow \mathcal{I}_{S/Q^4}\rightarrow 0.\]
Let $S\subseteq Q^4$ be a surface of type $Z_E ^{10}$. If $S$ is contained in a quadratic complex, then it has ideal sheaf resolution
\[0\rightarrow \ko_{Q^4}(-4)^{\oplus 2} \oplus \ko_{Q^4}(-3)\rightarrow \dot{\mathcal{E}}(-2)\oplus \ko_{Q^4}(-2)\oplus \ko_{Q^4}(-4) \rightarrow \mathcal{I}_{S/Q^4}\rightarrow 0.\]
If $S$ is not contained in a quadratic complex, then it has ideal sheaf resolution
\[0\rightarrow \mathcal{E}(-3)\oplus \mathcal{E}'(-3)\oplus \ko_{Q^4}(-4)\rightarrow \ko_{Q^4}(-3)^{\oplus 6} \rightarrow \mathcal{I}_{S/Q^4}\rightarrow 0.\]
Both cases occur.
\end{pr}
\begin{proof}
    See \cite[Proposition~4.3]{gross.degree10}.
\end{proof}  

In fact, only surfaces of type $Z_{E}^{10}$ not contained in a quadratic complex will be relevant to us. Thus, we have to make a clear distinction thereof.

\begin{defn}
Let $S\subseteq Q^4$ be a surface of type $Z_{E}^{10}$.
    We say that $S$ is of type $^I Z_{E}^{10}$ if $h^0(\mathcal{I}_{S/Q^4}(2))> 0$, and we say that $S$ is of type $^{II} Z_{E}^{10}$ if $h^0(\mathcal{I}_{S/Q^4}(2))=0$.
\end{defn} 

In other words, $S$ is of type $^I Z_{E}^{10}$ if and only if $S$ is contained in a quadratic complex, and $S$ is of type $^{II} Z_{E}^{10}$ if and only if $S$ is not contained in a quadratic complex.

Define the polynomial $P(t):=5t^2 - t +2 \in \QQ [t]$ and consider the Hilbert scheme ${Hilb}_{Q^4}^{P(t)}$ parametrizing closed subschemes of $Q^4$ with Hilbert polynomial $P(t)$. 

\begin{lem}\label{only two types of surfaces}
    Surfaces of type $Z_{B}^{10}$ and $Z_{E}^{10}$ are the only smooth surfaces in $Q^4$ with Hilbert polynomial $P(t)$.
\end{lem}
\begin{proof}
    Indeed, a smooth surface $S$ with Hilbert polynomial $P(t)=5t^2 -t +2=\frac{{H_S^2}  }{2}t^2 -\frac{H_S K_S}{2} t +\chi(\ko _S)$ has degree ${H_S^2} = 5\times 2=10$. Also, $S$ must necessarily have sectional genus $7$, since $2\pi -2={H_S^2}  + H_S K_S =10+2=12$. By the classification of non-degenerate smooth surfaces of degree $10$ in $\PP^4$ by K. Ranestad and S. Popescu (see \cite[Theorem~0.16]{Ranestad.SurfacesDeg1O}), no non-degenerate smooth surface of degree $10$ in $\PP^4$ can have sectional genus $7$. Moreover, a smooth surface of degree $10$ contained in $\PP^3$ cannot have sectional genus $7$ either, for a general hyperplane section is a smooth plane curve. Hence, $S$ is non-degenerate. Therefore, $S$ must be of type $Z_{B}^{10}$, $Z_{D}^{10}$, or $Z_{E}^{10}$ by \cite[\S 4 Table~1]{gross.degree10}. However, we can verify that only surfaces $S$ of type $Z_{B}^{10}$ and $Z_{E}^{10}$ satisfy $\chi(\ko _S)=2$ using the same table.
\end{proof}

\begin{lem}\label{Hilb is smooth at S type EII and type B} We have the following:
\begin{enumerate}
    \item Let $S\subseteq Q^4$ be a surface of type $ Z_B^{10}$. Then ${Hilb}_{Q^4}^{P(t)}$ is smooth at $[S]$ and the Zariski tangent space $T_{[S]}{Hilb}_{Q^4}^{P(t)}$ has dimension $36$.
    \item Let $S\subseteq Q^4$ be a surface of type $^{II} Z_{E}^{10}$. Then ${Hilb}_{Q^4}^{P(t)}$ is smooth at $[S]$ and the Zariski tangent space $T_{[S]}{Hilb}_{Q^4}^{P(t)}$ has dimension $38$.
\end{enumerate}
\end{lem}
\begin{proof}

    Let $S\subseteq Q^4$ be a surface of type $^{II}Z_{E}^{10}$. We can compute the cohomology groups $H^1(Q^4, N_{S/Q^4})$ and $H^2(Q^4 , N_{S/Q^4})$ as $Ext^2(\mathcal{I}_{S/ Q^4}, \mathcal{I}_{S/ Q^4})$ and $Ext^3(\mathcal{I}_{S/ Q^4} ,\mathcal{I}_{S/ Q^4})$, respectively, by \cite[Lemma~2.3]{arrondosols92}. We thus apply the functor $Hom(- , \mathcal{I}_{S/ Q^4} )$ to the resolution sequence of $\mathcal{I}_{S/ Q^4}$ given in Proposition \ref{resolution of ideal sheaf of S types E and B} to get a long exact sequence of right derived functors which will help us to compute them. Below we have an extract of the long exact sequences with the relevant spaces to consider. We have rewritten it in a cohomologial style using well-known properties of the functors $Ext^i(- , \mathcal{I}_{S/ Q^4} )$ and identifications $(\mathcal{E}(-3))^\ast = \mathcal{E}(4)$ and $(\mathcal{E'}(-3))^\ast =\mathcal{E'}(4)$.
\begin{align*}
    0&\rightarrow H^0(\mathcal{I}_{S/ Q^4}(3))^{\oplus 6} \rightarrow H^0((\mathcal{E}(4)\oplus \mathcal{E}'(4) )\otimes \mathcal{I}_{S/ Q^4})\oplus H^0(\mathcal{I}_{S/ Q^4}(4))\rightarrow H^0(N_{S/Q^4})\\
    &\rightarrow H^1(\mathcal{I}_{S/ Q^4}(3))^{\oplus 6} \rightarrow H^1((\mathcal{E}(4)\oplus \mathcal{E}'(4) )\otimes \mathcal{I}_{S/ Q^4})\oplus H^1(\mathcal{I}_{S/ Q^4}(4))\rightarrow H^1(N_{S/Q^4})\\
    &\rightarrow H^2(\mathcal{I}_{S/ Q^4}(3))^{\oplus 6} \rightarrow H^2((\mathcal{E}(4)\oplus \mathcal{E}'(4) )\otimes \mathcal{I}_{S/ Q^4})\oplus H^2(\mathcal{I}_{S/ Q^4}(4))\rightarrow H^2(N_{S/Q^4})\\
    &\rightarrow H^3(\mathcal{I}_{S/ Q^4}(3))^{\oplus 6} \rightarrow H^3((\mathcal{E}(4)\oplus \mathcal{E}'(4) )\otimes \mathcal{I}_{S/ Q^4})\oplus H^3(\mathcal{I}_{S/ Q^4}(4))\rightarrow 0
\end{align*}

    It is easy to see that $H^1(\mathcal{I}_{S/ Q^4}(3))=H^1(\mathcal{I}_{S/ Q^4}(4))=0$ using the resolution sequence of $\mathcal{I}_{S/ Q^4}$. Moreover, we have $H^i(\mathcal{I}_{S/ Q^4}(l))=0$ for all $l\geq 1$ and all $i\geq 2$ by \cite[Proposition~4.3]{gross.degree10}. Hence, $H^1(N_{S/Q^4})= H^1(\mathcal{E}(4)\otimes \mathcal{I}_{S/ Q^4}) \oplus H^1(\mathcal{E}'(4)\otimes \mathcal{I}_{S/ Q^4})$ and  $H^2(N_{S/Q^4})= H^2(\mathcal{E}(4)\otimes \mathcal{I}_{S/ Q^4}) \oplus H^2(\mathcal{E}'(4)\otimes \mathcal{I}_{S/ Q^4})$.

    In order to compute $H^1(\mathcal{E}(4)\otimes \mathcal{I}_{S/ Q^4})$ and  $H^2(\mathcal{E}(4)\otimes \mathcal{I}_{S/ Q^4})$ we tensor the resolution sequence of $\mathcal{I}_{S/ Q^4}$ with the vector bundle $\mathcal{E}$(4). The short sequence thus obtained remains exact and we can now consider its associated long exact sequence in cohomology
\begin{align*}
    0&\rightarrow H^0(\mathcal{E}\otimes\mathcal{E}^\ast) \oplus H^0(\mathcal{E}\otimes\mathcal{E'}^\ast) \oplus H^0(\mathcal{E})\rightarrow  H^0(\mathcal{E}(1))^{\oplus 6} \rightarrow H^0(\mathcal{E}(4)\otimes \mathcal{I}_{S/ Q^4}) \\
    &\rightarrow H^1(\mathcal{E}\otimes\mathcal{E}^\ast) \oplus H^1(\mathcal{E}\otimes\mathcal{E'}^\ast) \oplus H^1(\mathcal{E})\rightarrow  H^1(\mathcal{E}(1))^{\oplus 6} \rightarrow H^1(\mathcal{E}(4)\otimes \mathcal{I}_{S/ Q^4})\\
    &\rightarrow H^2(\mathcal{E}\otimes\mathcal{E}^\ast) \oplus H^2(\mathcal{E}\otimes\mathcal{E'}^\ast) \oplus H^2(\mathcal{E})\rightarrow  H^2(\mathcal{E}(1))^{\oplus 6} \rightarrow H^2(\mathcal{E}(4)\otimes \mathcal{I}_{S/ Q^4})\\
    &\rightarrow H^3(\mathcal{E}\otimes\mathcal{E}^\ast) \oplus H^3(\mathcal{E}\otimes\mathcal{E'}^\ast) \oplus H^3(\mathcal{E})\rightarrow  H^3(\mathcal{E}(1))^{\oplus 6} \rightarrow H^3(\mathcal{E}(4)\otimes \mathcal{I}_{S/ Q^4})\rightarrow \dots
\end{align*} 
    We have $H^i(\mathcal{E}\otimes\mathcal{E}^\ast) \oplus H^i(\mathcal{E}\otimes\mathcal{E'}^\ast) \oplus H^i(\mathcal{E})=H^i(\mathcal{E}(1))=0$ for $i=1$, $2$, and $3$ by \cite[\S 1.4]{arrondosols92}. Hence, $H^1(\mathcal{E}(4)\otimes \mathcal{I}_{S/ Q^4})=H^2(\mathcal{E}(4)\otimes \mathcal{I}_{S/ Q^4})=0$. Similarly, we get $H^1(\mathcal{E'}(4)\otimes \mathcal{I}_{S/ Q^4})=H^2(\mathcal{E'}(4)\otimes \mathcal{I}_{S/ Q^4})=0$. Consequently, $H^1(N_{S/Q^4})=H^2(N_{S/Q^4})=0$. In particular, ${Hilb}_{Q^4}^{P(t)}$ is smooth at $[S]$ and $dim(T_{[S]} {{Hilb}_{Q^4}^{P(t)}} )=dim (H^0(N_{S/Q^4}))=\chi(N_{S/Q^4})=38$ by Corollary \ref{euler characteristic of normal bundle of S}. Finally, the same arguments above can be applied to prove the claims about surfaces of type $ Z_{B}^{10}$ in the statement of the lemma.
\end{proof} 

\begin{pr}\label{open subset parametrizing surfaces of type II} There is a smooth and unirational open subset ${U}_{Q^4}\subseteq Hilb_{Q^4}^{P(t)}$ of dimension $38$ that parametrizes surfaces of type $^{II} Z_{E}^{10}$ in $Q^4$. 
\end{pr}
\begin{proof}
    The only smooth surfaces parametrized by ${Hilb}_{Q^4}^{P(t)}$ are those of type $Z_{B}^{10}$ and $ Z_{E}^{10}$ by Lemma \ref{only two types of surfaces}. By the semi-continuity theorem, there is an open subset of ${Hilb}_{Q^4}^{P(t)}$ characterized by the equation $h^0(\mathcal{I}_{S/Q^4} (2))= 0$. Observe that surfaces of type $Z_{B}^{10}$ and $^{II} Z_{E}^{10}$ are contained in this open subset, while surfaces of type $^{I} Z_{E}^{10}$ are not. Since the locus of smooth surfaces in ${Hilb}_{Q^4}^{P(t)}$ is open, and the tangent space dimensions differ from one class of surfaces to the other, we see that there are two disjoint smooth open subsets of ${Hilb}_{Q^4}^{P(t)}$ parametrizing surfaces of type $Z_{B}^{10}$ and $^{II} Z_{E}^{10}$, respectively. We thus have a smooth open subset ${U}_{Q^4}$ of $Hilb_{Q^4}^{P(t)}$ of dimension $38$ parametrizing surfaces of type $^{II} Z_{E}^{10}$ by Lemma \ref{Hilb is smooth at S type EII and type B}.

    We have $Hom(\mathcal{E}\oplus \mathcal{E}'\oplus \ko_{Q^4}(-1),\ko_{Q^4}^{\oplus 6}) = ({V_6} ^{\ast} \otimes {V_4}^{\ast}) \oplus ({V_6}^{\ast}  \otimes {V_4}) \oplus ({V_6}^{\ast}  \otimes \overset{2}{\wedge}{ V_4}^\ast)  $, where we set ${V_4}^\ast =H^0(Q^4,\mathcal{E}^\ast )$, $V_4 =H^0(Q^4,{\mathcal{E}'}^\ast )$, $\overset{2}{\wedge}{ V_4}^\ast =H^0(Q^4,\ko_{Q^4} (1) )$, ${V_6} ^{\ast}=H^0(Q^4, \ko_{Q^4}^{\oplus 6})$ for vector spaces of dimensions $4$ and $6$ denoted by $V_4$ and $V_6$, respectively. In other words, three tensors define a morphism $\mathcal{E}\oplus \mathcal{E}'\oplus \ko_{Q^4}(-1)\rightarrow \ko_{Q^4}^{\oplus 6}$. If the tensors are general, then the rank of the induced morphism drops in codimension two over a surface of type $Z_{E}^{10}$ by \cite[Proposition~6.3]{Fano4foldsK3type}. Conversely, for each surface $S$ of type $^{II} Z_{E}^{10}$ we have an associated exact sequence \[0\rightarrow \mathcal{E}\oplus \mathcal{E}'\oplus \ko_{Q^4}(-1)\rightarrow \ko_{Q^4}^{\oplus 6} \rightarrow \mathcal{I}_{S/Q^4}(3)\rightarrow 0\] by Proposition \ref{resolution of ideal sheaf of S types E and B}. Therefore, an open subscheme of the projective space $\PP:=\PP(({V_6} ^{\ast} \otimes {V_4}) \oplus ({V_6} ^{\ast} \otimes \overset{2}{\wedge}{ V_4}^\ast)\oplus ({V_6} ^{\ast} \otimes {V_4}^{\ast}) )$ maps onto the open subscheme ${U}_{Q^4}$ of $Hilb_{Q^4}^{P(t)}$. This defines a dominant rational map $\PP\dashrightarrow {U}_{Q^4}$, and thus unirationality is proved.
\end{proof}

Let $\mathcal{U}_{Q^4}:=\sfrac{ {U}_{Q^4}}{Aut(Q^4)}$ be the coarse moduli space of surfaces of type $^{II} Z_{E}^{10}$ in $Q^4$.
\begin{pr}\label{a general surface of type II has Picard number 1}
    Let $\mathcal{F}_{7}^{[2]}$ be the universal Hilbert scheme of zero-dimensional schemes of length $2$ in a polarized K3 surface of genus $7$. Then there is an open immersion $\mathcal{U}_{Q^4}\hookrightarrow \mathcal{F}_{7}^{[2]}$. In particular, the minimal model of a very general surface $S$ of type $^{II} Z_{E}^{10}$ has Picard rank $1$.
\end{pr}
\begin{proof}
    The locus of triples $(S_0, p, p')\in \mathcal{F}_{7}^{[2]}$, where $p$ and $ p'$ are distinct points on a K3 surface $S_0$ of genus $7$ polarized by $H_{S_0}$, and such that the linear system $|H_{S_0}-E_{p} -E_{p'}|$ embeds $Bl_{p ,p'}(S_0)$ into $\PP^5$, is an open set $V$, since very ampleness is an open condition. Observe that such surfaces in $\PP^5$ have Hilbert polynomial $P(t)=5t^2 -t+2$.
    
    Using semicontinuity arguments it is not difficult to see that there exists an open subset $U\subseteq Hilb_{\PP^5} ^{P(t)}$ parametrizing surfaces of type $^{II} Z_{E}^{10}$ in $\PP^5$. This open subset is invariant under the action of $PGL(6)$. Moreover, the moduli spaces $\mathcal{U}_{Q^4}=\sfrac{U_{Q^4}}{Aut(Q^4)}$ and $\sfrac{U}{PGL(6)}$ are isomorphic. Hence, the moduli space $\sfrac{Hilb_{\PP^5} ^{P(t)}}{PGL(6)}$ contains $\mathcal{U}_{Q^4}$ as an open subset. In particular, the morphism $\mu: V\rightarrow \sfrac{Hilb_{\PP^5} ^{P(t)}}{PGL(6)}$ given by $(S_0,p,p')\mapsto [Bl_{p ,p'}(S_0)\hookrightarrow\PP^5]$ can be restricted to the non-empty\footnote{See Remark \ref{remark}.} open set $\mu^{-1}(\mathcal{U}_{Q^4})\subseteq V\subseteq \mathcal{F}_{7}^{[2]}$. 
    
    On the other hand, if $S$ is a surface of type $^{II} Z_{E}^{10}$, then the linear system $|H_S +K_S|$ is base-point free by \cite[Proposition~2.2]{sommese.81}. It is easy to see that this linear system induces a morphism $S\rightarrow \PP^7$ that realizes the blow-up map of a K3 surface $S_0\subseteq \PP^7$ of degree $12$ at two points $p$ and $p'$. This allows us to define the morphism $ \nu:\mathcal{U}_{Q^4}\rightarrow \mathcal{F}_{7}^{[2]}$ by $[S]\mapsto (S_0,p ,p')$. Evidently, $\mu$ and $\nu$ are inverse to each other. Hence, $\mathcal{U}_{Q^4}\cong \mu^{-1}(\mathcal{U}_{Q^4})\subseteq \mathcal{F}_{7}^{[2]}$. The last claim follows from the well-known fact that a very general element of the moduli space of polarized K3 surfaces of genus $7$ has Picard rank $1$.
\end{proof}

\section{The family of Fano fourfolds K3-33}\label{section K3-33}

\subsection{Construction, notation, and conventions.} 

We describe the construction of a family of Fano fourfolds that induce special birational transformations of $Q^4$.

\begin{thm} Let $X$ be the zero locus of a general section of the vector bundle $\mathcal{U}_{G(2,4)}^\ast (0,1) \oplus \ko(1,1) \oplus \mathcal{Q}_{G(2,4)}(0,1)$ defined over the space $ G(2,4)\times\PP^5 $. Then $X$ is a smooth Fano fourfold of K3 type and the projections $\sigma : X\rightarrow G(2,4)$ and $\tau : X\rightarrow \PP^5 $ are both realized as the blow-up of a four-dimensional smooth quadric in $\PP^5$ along a surface of type $^{II} Z_{E}^{10}$. The corresponding minimal models are derived-equivalent K3 surfaces. 
 \end{thm}
 \begin{proof}
     See \cite[Fano K3-33]{Fano4foldsK3type}.
 \end{proof}

Any member $X$ of the family of fourfolds constructed in the previous theorem is called Fano fourfold K3-33. It is a so-called Fano manifold of K3-type, which in this case just means that $h^{3,1}(X)=1$.

The image of the second projection $\tau : X\rightarrow \PP^5$ is a smooth quadric $Q^4$. On the other hand, the first projection $\sigma : X\rightarrow G(2,4) $ is surjective. Since the Grassmannian $G(2, 4)$ is embedded into $\PP^5$ as a smooth quadric via the Plücker embedding, we will identify it with $Q^4$ as well. Let $S\subseteq Q^4$ (resp. $T\subseteq Q^4$) be the blow-up center of $\sigma$ (resp. $\tau$). We denote by $d$ and $d '$ (resp. $\delta$ and $\delta '$) the pair of skew $(-1)$-lines contained in $S$ (resp. $T$). Also, write $S_0$ (resp. $T_0$) for the minimal model of $S$ (resp. $T$).

We fix notation by means of the following diagram: 

\begin{center}
\begin{tikzcd}
                               &             & E \arrow [r, hook,  "j_1"] \arrow [ld, "\sigma_|"'] & X \arrow [ld, "\sigma"'] \arrow [rd, "\tau"] & E' \arrow [l,  hook'   ,  "j_2"'] \arrow [rd, "\tau_|"] &             &                     \\
d \cup d ' \arrow [r, hook] & S \arrow [r, hook] & Q^4                      &                         & Q^4                      & T \arrow [l, hook' ] & \delta \cup \delta ' \arrow [l, hook']
\end{tikzcd}
\end{center}

The K3 surfaces $S_0$ and $T_0$ are derived-equivalent, i.e., they are Fourier-Mukai partners. We will show that $S_0$ is not isomorphic to $T_0$ if their Picard rank is $1$. This holds for almost all members of the family K3-33, according to the following proposition.

\begin{pr}
    For a very general member $X$ of the family Fano fourfolds K3-33 we have $rk(Pic(S_0))=rk(Pic(T_0))=1$. 
\end{pr}
\begin{proof} 
    We have $H^0(G(2,4)\times\PP^5 , \mathcal{U}_{G(2,4)}^\ast (0,1) )={V_4}^{\ast}\otimes {V_6} ^{\ast} $, $H^0(G(2,4)\times\PP^5,\ko (1,1))=\overset{2}{\wedge}{ V_4}^{\ast} \otimes {V_6} ^{\ast} $, and $H^0(G(2,4)\times\PP^5, \mathcal{Q}_{G(2,4)} (0,1))={V_4} \otimes {V_6} ^{\ast} $. Hence, the family of Fano fourfolds K3-33 is parametrized by an open subscheme of $\PP=\PP(({V_6} ^{\ast} \otimes {V_4}) \oplus ({V_6} ^{\ast} \otimes \overset{2}{\wedge}{ V_4}^\ast)\oplus ({V_6} ^{\ast} \otimes {V_4}^{\ast}) )$. In fact, in \cite[Fano K3-33]{Fano4foldsK3type} the surface $S$ corresponding to a fourfold $X$ belonging to the family is defined as the degeneracy locus of the morphism $\mathcal{E}\oplus \mathcal{E'}\oplus \ko_{Q^4} (-1)=\mathcal{U}_{G(2,4)}\oplus \mathcal{Q}_{G(2,4)}^{\ast}\oplus \ko_{G(2,4)}(-1)\rightarrow {V_6} ^{\ast}\otimes \ko_{G(2,4)} =\ko_{Q^4}^{\oplus 6}$ that is induced by the same element of $\PP$ defining $X$. Since the rational map $\PP\dashrightarrow U_{Q^4}$ mapping a Fano fourfold $X$ of type K3-33 to its corresponding surface $S$ is dominant (see the proof of Proposition \ref{open subset parametrizing surfaces of type II}), the claim follows by Proposition \ref{a general surface of type II has Picard number 1}.
\end{proof}

In view of the previous proposition, we will assume from now on that $S_0$, or equivalently, $T_0$ has Picard rank $1$. In order to show that $S_0\ncong T_0$, we start by applying the blow-up formula for singular cohomology to $X$ to get a composition of Hodge isometries:
\[H^{4}(Q^4,\ZZ)\oplus H^{2}(S,\ZZ)\xrightarrow  {\sigma^{\ast}+j_{1\ast}{\sigma_{|}}^{\ast}} H^{4}(X,\ZZ) \xleftarrow  {\tau^{\ast}+j_{2\ast}{\tau_{|}}^{\ast}}H^{4}(Q^4,\ZZ)\oplus H^{2}(T,\ZZ)\]

We denote by $H$ (resp. $H'$) the pullback $\sigma ^\ast (H)$ (resp. $\tau ^\ast (H)$) on $X$. We use the same symbol to denote a codimension $2$ cycle in $X$ coming from $S$ (resp. $T$) via $j_{1\ast}{\sigma_{|}}^{\ast}$ (resp. $j_{2\ast}{\tau_{|}}^{\ast}$). For example, $j_{1\ast}{\sigma_{|}}^{\ast}(d)$ is simply denoted by $d$. On the other hand, we have $H^{4}(Q^4 ,\ZZ)= CH^2(Q^4)=\ZZ P_1 \oplus \ZZ P_2$. We denote by $P_1$ and $P_2$ the pullbacks $\sigma ^\ast (P_1)$ and $\sigma ^\ast (P_2)$ on $X$. The notation $P_1 '=\tau ^\ast (P_1)$ and $P_2 '=\tau ^\ast (P_2)$ is not used, although it is left implicit. We thus have an identification $H^{4}(X,\ZZ)=H^{4}(Q^4,\ZZ)\oplus H^{2}(S,\ZZ)=H^{4}(Q^4,\ZZ)\oplus H^{2}(T,\ZZ)$. Observe that this lattice is unimodular, although it is not even.

Denote by $H_S$ (resp. $H_T$) the polarization of degree $10$ of $S$ (resp. $T$). Recall that the linear system $|H_S +d+d'|$ (resp. $|H_T +\delta+\delta '|$) induces a morphism to $\PP^7$ that realizes the blow-up map $S\rightarrow S_0$ (resp. $T\rightarrow T_0$). The induced polarization of degree $12$ on $S_0$ (resp. $T_0$) is denoted by $H_{S_0}$ (resp. $H_{T_0}$). Thus, $NS(S)\subseteq H^{2}(S,\ZZ)$ is freely generated by the divisor classes $H_S +d +d '$, $d$, and $d '$. We have a similar statement for $T$.

The algebraic lattice $A_1$ of $H^{4}(X,\ZZ)=H^{4}(Q^4,\ZZ)\oplus H^{2}(S,\ZZ)$ is by definition the subgroup generated by algebraic cycles of codimension $2$ in $X$, i.e., $A_1=H^{4}(Q^4,\ZZ)\oplus NS(S)$. The orthogonal lattice of $A_1$ in $H^{4}(X,\ZZ)$ with respect to the intersection form is evidently $T(S_0)$, the transcendental lattice of the K3 surface $S_0$. 

The intersection matrix of the lattice $A_1$ is 

\begin{center}
   \begin{tabular}{c| c c c c c}
             & $P_1$ & $P_2$ & $H_S +d +d '$ & $d$ & $ d '$ \\
             \hline
     $P_1$ & $1$       &  $0$       & $0$                      &  $0$      &  $0$       \\
    $ P_2$ & $0$        & $1$        & $0$                      &  $0$      &  $0$       \\
    $ H_S +d +d '$ & $0$       & $0$        &      $12$                 &  $0$      &  $0$       \\
     $d  $     &  $0$       & $0$                     &$0$ &     $-1$   &  $0$       \\
     $d  '$     & $0$        & $0$                      &$0$ & $0$      &        $-1$ 
    \end{tabular} 
\end{center}
    
We can define an algebraic lattice $A_2$ in $H^{4}(X,\ZZ)=H^{4}(Q^4,\ZZ)\oplus H^{2}(T,\ZZ)$ in the same way as before, and for the induced isometry of Hodge structures $\phi_A \oplus \phi_T : A_1\oplus T(S_0)\rightarrow A_2 \oplus T(T_0)$ we have the following commutative diagram of discriminant groups by \cite[Corollary~1.5.2]{nikulin}:

\begin{center}
\begin{tikzcd}
dA_1 \arrow[d, "\cong"'] \arrow[r, "\phi_A"] & dA_2 \arrow[d, "\cong"] \\
dT(S_0)(-1) \arrow[r, "\phi_T"]              & dT(T_0)(-1)            
\end{tikzcd}
\end{center}

It is clear that $dA_1$ is isomorphic to $\ZZ_{12}$, with generator $\frac{H_S +d+d'}{12}\text{ mod }A_1$. Moreover, if $b_{S}:=|H_S+d+d'|:S\rightarrow S_0$ is the blow-up map, then $dA_1=d(H^4(Q^4 , \ZZ) \oplus NS(S))=d(\ZZ P_1 \oplus \ZZ P_2) \oplus d(b_{S}^{\ast}NS(S_0))\oplus d(\ZZ d\oplus\ZZ d')=d(b_{S}^{\ast}NS(S_0))$. Hence, $b_{S}^{\ast}: dNS(S_0)\rightarrow d(b_{S}^{\ast}NS(S_0))=dA_1$ is an isometry given by $\frac{H_{S_0}}{12}\text{ mod }NS(S_0) \mapsto \frac{H_S+d+d'}{12}\text{ mod }A_1$. An analogous statement holds for $A_2$.

\subsection{Geometric action on cycles.}

Next, we want to understand how cycles coming from $\sigma$ and cycles coming from $\tau$ relate to each other in $H^4(X, \ZZ)$. This is fundamental, for then we can find a lattice-theoretic property of the Hodge isometry that will ultimately help us to prove that $S_0$ is not isomorphic to $T_0$.

\begin{pr}\label{intersection of codim 2 cycles in k3-33}
The relations $K_X=-4H+E=-4H'+E'=-H-H'$ hold in $Pic(X)$. In particular, we have the cubo-cubic relations $H'=3H -E\text{ and }H=3H' -E'$. Moreover, the following relations also hold in $CH^2(X)$:

\begin{enumerate}
\item $H E=H_S\text{ and }H' E'=H_T$.
\item $3H ^2 -H_S=3{H'} ^2 -H_T$.
\item $E^2=-5H ^2+d+d ' + 4H_S\text{ and } {E'}^2=-5{H'} ^2+\delta+\delta ' + 4H_T$.
\end{enumerate}
\end{pr}

\begin{proof}
    The canonical bundle of $X$ is given by $ K_X=[K_{G(2,4) \times \PP^5}\otimes \overset{5}{\bigwedge}( \mathcal{U}_{G(2,4)}^\ast (0,1)\oplus \ko(1,1) \oplus \mathcal{Q}_{G(2,4)}(0,1))]_{|X}
    = - H - H'$. We also have $K_X=-4H +E=-4H' +E'$ by the blow-up formula for the canonical bundle. On the other hand, (\textit{a}) is easily derived from \cite[Example~8.3.9 iii)]{Fulton} and (\textit{b}) follows from a straightforward computation using (\textit{a}) and the equations $(-4H+E)(-H-H')=K_X ^2=(-4H'+E')(-H-H' )$, $H ^2=3H H' - E'H$ and ${H'} ^2=3H H' - EH'$. We prove (\textit{c}) now. We have $H ^2=P_1 +P_2$ in $CH^2(Q^4)$. Moreover, $S=5H^2=5P_1 +5P_2$ according to Theorem \ref{classif deg 10 surfaces thm}. Also, we have $c_1(Q^4)=4H$, $c_2(Q^4)=7H^2$ and $c_1(S)=-d -d '$. Hence, $c_2(X)=7H ^2 + d +d '- E ^2 = 12H ^2 -4H_S$ by \cite[Example~15.4.3]{Fulton}. From the last equality we obtain the same expression for $E ^2 $ as in the statement. By symmetry, we can derive an analogous expression for ${E'}^2$.
\end{proof} 

The morphism $\phi _A$ induced from the Hodge isometry of $H^4(X,\ZZ)$ relates geometrically the algebraic lattices $A_1$ and $A_2$ as follows. 

\begin{lem}\label{mult by 5}
$\phi_A:dA_1\rightarrow dA_2$ is given by multiplication by $5$.
\end{lem}

\begin{proof}
    We systematically use the relations described in Proposition \ref{intersection of codim 2 cycles in k3-33} above. 

In $H^4(X,\ZZ)$ we have

\begin{equation*}
    \begin{split}
        {H'} ^2&=(3H - E)^2\\
        &=9H ^2 -6H_S + E ^2\\
        &=9H ^2 -6H_S + (-5H ^2+d+d ' + 4H_S)\\
        &=4H ^2 - 2H_S +d + d '\\
        &=4H ^2 -3H_S + (H_S +d +d ')\\
        &= 9H ^2 -3H_S -5H ^2 + (H_S +d +d ')\\
        &=3(3H ^2 -H_S)-5H ^2 + (H_S +d +d ')\\
        &=3(3{H'} ^2 -H_T)-5H ^2 + (H_S +d +d ')\\
        &= 9{H'} ^2 -3H_T -5H ^2 + (H_S +d +d ')
    \end{split}
\end{equation*}

From this we get 

\begin{equation*}
    \begin{split}
        (H_S +d +d ')&=-8{H'}^2+3H_T+5H ^2\\
        &=-8{H'}^2+3H_T+5(3{H'} -E')^2\\
        &= -8{H'}^2+3H_T+5(9{H'} ^2 -6H_T + {E'} ^2)\\
        &=-8{H'}^2+3H_T+5(9{H'} ^2 -6H_T +(-5{H'}^2+\delta+\delta ' + 4H_T))\\
        &=-8{H'}^2+3H_T+5(4{H'} ^2 -2H_T + \delta+\delta ')\\
        &=-8{H'}^2+3H_T+20{H'} ^2-10H_T +5(\delta+\delta ')\\
        &=12{H'} ^2- 7H_T +5(\delta+\delta ')\\
        &=12{H'} ^2- 7H_T +5(\delta+\delta ') -5H_T+5H_T\\
        &=12{H'} ^2- 12H_T +5(H_T+\delta+\delta ')
    \end{split}
\end{equation*}

Therefore, 

\begin{equation*}
    \begin{split}
        \phi _A\left(\frac{H_S +d+d'}{12}\text{ mod }A_1\right)&=\frac{12{H'} ^2 - 12H_T + 5(H_T +\delta+\delta ')}{12}\text{ mod }A_2\\
        &= {H'} ^2 - H_T + 5\left(\frac{H_T +\delta+\delta '}{12}\right)\text{ mod }A_2\\
        &= 5\left(\frac{H_T +\delta+\delta '}{12}\right)\text{ mod }A_2. 
    \end{split}
\end{equation*}
This completes the proof of the lemma.
\end{proof} 

The main result of this section is

\begin{thm}\label{T0 is not S0}
Assume $S_0$, or equivalently, $T_0$ has Picard rank $1$. Then $S_0 \ncong T_0$.
\end{thm}
\begin{proof}

    Let $\theta: S_0 \rightarrow T_0$ be an isomorphism. It induces a Hodge isometry $H^2(T_0, \ZZ)\rightarrow H^2(S_0,\ZZ)$, which can be decomposed as $ \theta _A \oplus\theta _T : NS(T_0)\oplus T(T_0)\rightarrow NS(S_0)\oplus T(S_0)$. Moreover, since these surfaces have Picard rank $1$, the degree $12$ polarizations $H_{S_0}$ and $H_{T_0}$ are interchanged by $\theta_A$.

    The unimodularity of the K3 lattice induces the following diagram of discriminant groups by \cite[Corollary~1.5.2]{nikulin}:

\begin{center}
\begin{tikzcd}
dNS(T_0) \arrow[d, "\cong"'] \arrow[r, "\theta_A"] & dNS(S_0) \arrow[d, "\cong"] \\
dT(T_0)(-1) \arrow[r, "\theta_T"]              & dT(S_0)(-1)            
\end{tikzcd}
\end{center}

    Evidently, we can identify the isometry $\theta_T :dT(T_0)(-1)\rightarrow dT(S_0)(-1)$ with the isometry $ dA_2\overset{b_{T}^{\ast}}{\leftarrow}dNS(T_0)\overset{\theta_A}{\rightarrow} dNS(S_0)\overset{b_{S}^{\ast}}{\rightarrow}dA_1$ given by $\frac{H_T +\delta +\delta'} {12}\text{ mod }A_2\mapsto \frac{H_S+d+d'}{12}\text{ mod }A_1$. We will continue to denote the latter by $\theta_A$. Now observe that the composition $\theta _T\phi _T$ gives a Hodge isometry of the transcendental lattice $T(S_0)$ of $S_0$. However, $Id$ and $-Id $ are the only Hodge isometries of transcendental lattices of algebraic K3 surfaces with odd Picard rank by \cite[Lemma~4.1]{oguiso.almost-primes}. Consequently, $ \theta _T \phi _T = \pm Id$, and this equality still holds if we view it as an automorphism of $dT(S_0)$. But then the map $\theta_A \phi_A\cong \theta _T \phi _T$ is both $\pm Id$ and multiplication by $5$ in $dA_1\cong \frac{\ZZ}{12\ZZ}$, by Lemma \ref{mult by 5}. Contradiction.
\end{proof} 
\section{Special birational transformations}\label{section special bir trans}

\subsection{Basic facts on special rational maps.}

Let $\varphi: Y \dashrightarrow Z$ be a rational map projective varieties, where we assume that $Y$ is factorial, hence normal. Then we can describe $\varphi$ by a linear system. Indeed, by factoriality there exists a unique invertible sheaf $\mathcal{L}$ that extends the pullback $({\varphi}_{|_{dom(\varphi)}}) ^{\ast}(\ko_Z (1))$ to all of $Y$. Clearly, we can also extend the sections $({\varphi}_{|_{dom(\varphi)}}) ^{\ast}(s)$, $s\in H^0(Z, \ko_Z (1))$, to sections of $\mathcal{L}$. The corresponding map $\varphi ^{\ast} : H^0(Z, \ko_Z (1))\rightarrow H^0(Y, \mathcal{L})$ is injective and its image defines a linear system in $|H^0(Y, \mathcal{L})|$ that induces $\varphi$ up to a proyectivity. The base locus scheme $\mathfrak{B}$ of $\varphi$ is defined as the scheme-theoretic intersection of any collection of divisors that make up a basis of this linear system. We say that $\varphi$ is special if $\mathfrak{B}$ is a smooth subvariety of $Y$. 

Recall that the graph $\Gamma_\varphi$ of the rational map $\varphi$ is defined as the closure of the graph of the morphism ${\varphi}_{|_{dom(\varphi)}}$ in $Y\times Z$. It has the universal property that characterizes the blow-up of $Y$ along the closed subscheme $\mathfrak{B}$ (or more precisely, along the ideal sheaf corresponding to $\mathfrak{B}$). Denote by $\sigma: Bl_{\mathfrak{B}} (Y)\rightarrow Y $ the blow-up morphism. Then there exists a unique morphism $\tau:Bl_{\mathfrak{B}} (Y)\rightarrow Z$ such that $(\Gamma_ \varphi, \pi_Y, \pi_Z )\cong ( Bl_{\mathfrak{B}} (Y), \sigma, \tau)$. In fact, $ \tau$ is induced up to a proyectivity by the complete linear system $ |\sigma^\ast \mathcal{L}\otimes \ko(-E) |$, where $E$ is the exceptional divisor of $\sigma$. We even say that the rational map $\varphi$ is resolved by the morphisms $\sigma$ and $\tau$ via the diagram $Y \xleftarrow {\sigma} Bl_{\mathfrak{B}} (Y) \xrightarrow {\tau}Z$. Moreover, if $\varphi$ is a birational map and $Z$ is factorial, then $\tau$ is the blow-up of $Z$ along the base locus scheme $\mathfrak{B}'$ of the inverse map $\varphi ^{-1}$, for the isomorphism $Y\times Z \cong Z\times Y$ restricts to the graphs $\Gamma_\varphi$ and $\Gamma_ {\varphi^{-1}}$. Furthermore, the subschemes $\mathfrak{B}$ and $\mathfrak{B}'$ have codimension at least $2$ by \cite[Ch. V, Lemma~5.1]{hartshorne}.

\subsection{Numerical constraints.}

Let $k>2$ and let $\varphi:Q^k\dashrightarrow Q^k$ be a special birational transformation whose base locus scheme is a smooth variety $\mathfrak{B}$. Set $X:=Bl_\mathfrak{B} (Q^k)$. As before, we denote by $\sigma$ the blow-up morphism $X\rightarrow Q^k$, by $E$ its exceptional divisor, and by $\tau$ the morphism extending $\varphi$ to $X$. We will also keep denoting by $H$ (resp. $H'$) the pullback $\sigma ^\ast (H)$ (resp. $\tau ^\ast (H)$) on $X$. Observe that since $\mathfrak{B}$ is smooth and $Pic(Q^k)=\ZZ H$, we can decompose the Picard group of $X$ as $Pic(X)=\ZZ H \bigoplus \ZZ E$. 

Let $\mathfrak{B}'$ be the base locus scheme of $\varphi ^{-1} $, which is not necessarily smooth, reduced, or irreducible. Then $\tau $ is the blow-up of $Q^k$ along $\mathfrak{B}'$ by our discussion above. Moreover, if $E':=\tau ^{-1 }(\mathfrak{B}')$ is the exceptional divisor of $\tau$, then \cite[Proposition~1.3]{einshepherdbarron} implies that $E'$ is an irreducible divisor. Also, we have $E'=\mathsf{b}(E'_{red})$ as effective Cartier divisors for some positive integer $\mathsf{b}$.

According to our setting, relations $H'=nH-E$ and $H=m H' - E'=mH'-bE'_{red}$ hold in $Pic(X)$ for some positive integers $n$ and $m$. Write $r:=dim(\mathfrak{B})$ and $r':=dim(\mathfrak{B}')$.

\begin{pr}\label{equations in PicX} We have $\mathsf{b}=1$. In other words, $E'$ is reduced. Moreover, we have
\begin{enumerate}
    \item If $\Omega=Q^k -Sing(\mathfrak{B}'_{red})$, then $\mathfrak{B}'\cap \Omega=\mathfrak{B}'_{red} \cap \Omega$, i.e., $\mathfrak{B}'$ is generically reduced.
    \item $Pic(X)=\ZZ H \oplus \ZZ E=\ZZ H' \oplus \ZZ E'$.
    \item $K_X= -kH+(k-r-1)E= -kH'+(k-r'-1)E'$ in $Pic(X)$.
    \item $E=(mn-1)H'-nE'$ and $E'=(mn-1)H-mE$ in $Pic(X)$.
\end{enumerate} 
\end{pr}
\begin{proof}
    We have $\mathsf{b}E'_{red}=mH'-H=(mn-1)H-mE$. If $l\subseteq Q^k$ is a line and $f$ is a fiber of $E\rightarrow \mathfrak{B}$, then $\mathsf{b}\vert ((mn-1)H-mE)l =mn-1$ and $\mathsf{b}\vert ((mn-1)H-mE)f=m$. Thus, $\mathsf{b}=1$ follows. Finally, observe that the same kind of arguments used to prove \cite[Proposition~2.1]{einshepherdbarron} also prove (\textit{a}), (\textit{b}), (\textit{c}), and (\textit{d}).
\end{proof}

\begin{cor}\label{formulas for n,m,r,r'}
    We have $n=\frac{r'+1}{k-r-1}$ and $m=\frac{r+1}{k-r'-1}$, or equivalently, $r=(k-r'-1)m-1$ and $r'=(k-r-1)n-1$.
\end{cor}
\begin{proof}
    By Proposition \ref{equations in PicX} we have $-kH+(k-r-1)E=-kH'+(k-r'-1)E'=-k(nH-E)+(k-r'-1)((mn-1)H-mE)$, or equivalently, $-kH-(r+1)E=-knH+(k-r'-1)((mn-1)H-mE)=(-kn+(k-r'-1)(mn-1))H-(k-r'-1)mE=(-kn+(k-r'-1)mn-(k-r'-1))H-(k-r'-1)mE$. In particular, $(k-r'-1)m=r+1$ and $-k=-kn+(k-r'-1)mn-(k-r'-1)=-kn+(r+1)n-(k-r'-1)=-(k-r-1)n-(k-r'-1)$, i.e., $r'=(k-r-1)n-1$.
\end{proof}

\begin{pr}\label{values of n,m,r,r'}
    If $k=3$, then $(n,m)=(2,2)$, i.e., $\varphi$ is a quadro-quadric transformation of $Q^3$, and $(r,r')=(1,1)$; if $k=4$, then $(n,m)=(3,3)$, i.e., $\varphi$ is a cubo-cubic transformation of $Q^4$, and $(r,r')=(2,2)$.
\end{pr}
\begin{proof}
    This follows from Corollary \ref{formulas for n,m,r,r'} and the fact that $\mathfrak{B}$ and $\mathfrak{B}'$ have codimension at least $2$ in $Q^k$.
\end{proof}

In the following, $\mathfrak{B}$ will be denoted by $C$ if $k=3$, and by $S$ if $k=4$. Let $\mathsf{d}$ be the degree of $C$ and $\mathsf{g}$ its genus, and let $\mathfrak{d}$ be the degree of $S$ and $(\mathfrak{a},\mathfrak{b})$ its bidegree.

\begin{lem} \label{intersection theory} We list the intersection numbers $H^iE^{k-i}$, where $k=3,4$.\\
If $k=3$, then 
\begin{enumerate}
    \item $H^3=2$.
    \item $H^2 E=0$.
    \item $H E^2=-\mathsf{d}$.
    \item $E^3= -2\mathsf{g}+2-3\mathsf{d}$.
\end{enumerate}
If $k=4$, then 
\begin{enumerate}
    \item $H^4=2$.
    \item $H^3 E=0$.
    \item $H^2 E^2=-\mathfrak{d}$.
    \item $H E^3= -H_S K_S -4\mathfrak{d}$.
    \item $E^4=-9\mathfrak{d} -c_2 (S)-4H_S K_S$.
\end{enumerate}
\end{lem}

\begin{proof}
    See \cite[Lemma~2.2.14 $(ii)$ $b)$]{AGV.Fano} for the case $k=3$, and \cite[Lemma~2.3 $(ii)$]{prokhorov.zaidenberg-examples} for the case $k=4$.
\end{proof}

\begin{pr}\label{equationsofS} If $k=3$, then $C$ is a smooth rational quartic curve; if $k=4$, then the degree $\mathfrak{d}$ of the surface $S$ is less than $18$. Moreover, the following equations hold:
\begin{enumerate}
    \item $H_S K_S =5\mathfrak{d}-48$.
    \item $c_2 (S) = 25\mathfrak{d} -224$.
    \item $\pi=3\mathfrak{d} -23$.
    \item $12 \chi (\ko_S)= K_S ^2 +c_ 2(S)= \mathfrak{d} ^2 +23\mathfrak{d} -2\mathfrak{a}\mathfrak{b} -256$.
\end{enumerate}
\end{pr}
\begin{proof}
    Let $k$ be arbitrary. We have $2=H'^k=(nH-E)^k$ since $Q^k$ has degree $2$. Now let us consider a general complete intersection of $k-1$ hyperplane sections of $Q^k$. This is a smooth conic curve, and its strict transform with respect to $\tau$ is mapped onto a curve contained in the source $Q^k$ via $\sigma$. The degree of this curve is $2m =(m H' - E')(H')^{k-1}=H(H')^{k-1} =H(nH-E)^{k-1} $. In the following, we will use Proposition \ref{values of n,m,r,r'} and Lemma \ref{intersection theory}. Assume $k=3$. Then $2=(2H-E)^3=14-3\mathsf{d} +2\mathsf{g}$ and $4=H(2H-E)^2=8-\mathsf{d}$. It follows that $\mathsf{d} =4$ and $\mathsf{g}=0$, and thus $C$ is a rational quartic curve. Now assume $k=4$. We consider a general complete intersection of two hyperplane sections of $Q^4$. This is a smooth quadric surface in $Q^4$, and its strict transform with respect to $\tau$ is mapped onto a surface contained in the source $Q^4$ via $\sigma$. The degree of this surface is $H^2 H'^2=H^2(3H-E)^2=2(3)^2 -\mathfrak{d}  >0$, and thus $\mathfrak{d} <18$. Items (\textit{a}) and (\textit{b}) follow from the equations $2=(3H-E)^4$ and $6=H(3H-E)^3$. To prove (\textit{c}), we use the adjunction formula $2\pi -2 = H_S K_S + H_ S^2$ together with (\textit{a}). Finally, (\textit{d}) follows from Noether's formula and Corollary \ref{square of K_S}.
\end{proof}

We want to understand the geometry of the curve $C$ and the surface $S$, and ultimately provide a model for them. The following proposition gives indeed a model for $C$ and also basic qualitative information about the embedding of the surface $S$ into $Q^4$.

\begin{pr} \label{Sisnotinaquadraticcomplex} If $k=3$, then $H ^0(Q^3, \mathcal{I}_{C/Q^ 3} (2))=\CC ^5$ and $H^0(Q^3, \mathcal{I}_{C/Q^ 3} (1))=0$. In particular, $C$ is a rational normal quartic curve in $Q^3$; if $k=4$, then $H ^0(Q^4, \mathcal{I}_{S/Q^ 4} (3))=\CC ^6$ and $H^0(Q^4, \mathcal{I}_{S/Q^ 4} (1))=H^0(Q^4, \mathcal{I}_{S/Q^ 4} (2))=0$. In particular, $S$ is not a complete intersection, it is non-degenerate, and it is not contained in a quadratic complex of $Q^4$.  
\end{pr}
\begin{proof}
    We have $\CC ^{k+2}= \varphi ^{\ast}(H^0(Q^k ,\ko_{Q^k}(1)))=H^0(Q^k , \mathcal{I}_{\mathfrak{B}/Q^ k}(n)) =H^0 (X, \ko_X (nH-E))$ by \cite[Theorem~7.1.3]{dolgachev.classical}. Now assume $j<n $. If $0\neq s\in H^0(Q^k, \mathcal{I}_{\mathfrak{B}/Q^ k} (j))$, then we can define $k+2$ independent non-zero sections $sx_0 ^{n-j}, sx_1 ^{n-j}, \dots, sx_{k+1} ^{n-j}\in H^0(Q^k, \mathcal{I}_{\mathfrak{B}/Q^ k} (n))$, where $x_0,\dots, x_{k+1}$ are coordinates giving a basis of $ H^0(Q^k ,\ko_{Q^k}(1))$. But then we see that the codimension $1$ zero locus $V(s)$ must be contained in the base locus scheme $\mathfrak{B}$. This is absurd. In particular, if $k=3$ or $4$, then we get the cohomology groups in the statement. Also notice that $C$ is a rational normal curve since $C\subseteq \PP^4$ has degree $4$, it is isomorphic to $\PP^1$, and is non-degenerate. The statements about $S$ are clear.
\end{proof}

At this point it is possible to give a classification of special birational maps of the three-dimensional smooth quadric $Q^3$.

\begin{proof}[(Proof of Theorem \ref{classificationSpecialtransQ3})]
    The first claim follows from Proposition \ref{values of n,m,r,r'} and Proposition \ref{Sisnotinaquadraticcomplex}. The converse statement is essentially proved in \cite[Lemma~2.2]{kuznetsov2022rationality}. We just remark that smoothness of the curve $\mathfrak{B}'_{red}$ implies $\mathfrak{B}'=\mathfrak{B}'_{red}$ by (\textit{a}) in Proposition \ref{equations in PicX}. Thus, $\mathfrak{B}'$ is also a rational normal quartic curve by symmetry.
\end{proof}

The four-dimensional case needs more work. As a first step, we find the degree of the surface $S$ using some results from the theory of congruences of lines in $\PP^3$.

\begin{pr}\label{S is degree 10} $S$ is a surface of degree $\mathfrak{d} =10$.
\end{pr}
\begin{proof}
    We have $3\mathfrak{d} -23=\pi >0$ by (\textit{c}) in Proposition \ref{equationsofS}. In particular, we have $8\leq \mathfrak{d} \leq 17$. If $\mathfrak{d} =8$ or $9$, we consider all pairs $(\mathfrak{a},\mathfrak{b})$ of non-negative integers $\mathfrak{a}$ and $\mathfrak{b}$ such that $\mathfrak{a}+\mathfrak{b}=\mathfrak{d}$. For all such pairs we necessarily have $12\vert K_S ^2 +c_ 2= \mathfrak{d} ^2 +23\mathfrak{d} -2\mathfrak{a}\mathfrak{b} -256$ by (\textit{d}) in Proposition \ref{equationsofS}. However, it can be easily checked that this divisibility condition never holds. On the other hand, if $\mathfrak{d} =13$, $14$ or $15$, we get $\pi > \pi_ 1 (\mathfrak{d})$. This contradicts Theorem \ref{maximalsectionalgenus}. Finally, if $\mathfrak{d} =11$, $12$, $16$ or $17$, we get $\pi = \pi_ 1 (\mathfrak{d})$. Then Theorem \ref{lowerboundofp.a} implies that $S$ is contained in a quadratic complex. This not possible by Proposition \ref{Sisnotinaquadraticcomplex}. 
\end{proof}

Using the classification of smooth surfaces of degree $10$ in the four-dimensional smooth quadric, we can easily identify the geometric type of $S$. It turns out that $S$ is a non-minimal K3 surface. 

\begin{pr}\label{S is a blow of a K3} $S$ is a surface of type $^{II} Z_{E}^{10}$. In particular, $S$ contains two skew $(-1)$-lines whose contraction defines a K3 surface of genus $7$.
\end{pr}
\begin{proof}
    Recall that $S$ is non-degenerate and it is not contained in a quadratic complex of $Q^4$ by Proposition \ref{Sisnotinaquadraticcomplex}. By substituting $\mathfrak{d} =10$ in the formulas of Proposition \ref{equationsofS}, we get the following numerical data: $H_S K_S=2 $, $ c_2(S)=26 $, and $ \pi =7$. These values can only belong to a surface of type $ Z_{E}^{10}$ by Theorem \ref{classif deg 10 surfaces thm}, which also has bidegree $(\mathfrak{a},\mathfrak{b})=(5,5)$ and $K_S^2= -2$ (see \cite[\S 4 Table~1]{gross.degree10}, where the main invariants of non-degenerate surfaces of degree up to $10$ in $Q^4$ are listed). Therefore, $S$ is a surface of type $^{II} Z_{E}^{10}$.
\end{proof}

The next proposition summarizes the characteristics of $\varphi$ that we have found so far in the four-dimensional case.

\begin{pr}\label{characterization of special bir trans of 4-quadric}
 Let $\varphi:Q^4\dashrightarrow Q^4$ be a special birational transformation of $Q^4$. Then $\varphi$ is a cubo-cubic transformation, and its base locus scheme is a surface $S$ of type $^{II} Z_{E}^{10}$.
\end{pr}
\begin{proof}
    This is the content of Proposition \ref{values of n,m,r,r'} and Proposition \ref{S is a blow of a K3}.
\end{proof}

\subsection{Effective classification.}
We are now interested to know if the converse of Proposition \ref{characterization of special bir trans of 4-quadric} holds, i.e., we ask the following question:

\textit{Do all surfaces of type $^{II} Z_{E}^{10}$ induce special self-birational transformations of the smooth quadric fourfold?}

For the moment, we can prove a weaker statement. 

\begin{pr}\label{S induces a special birational transformation onto a quadric} Let $S\subseteq Q^4$ be a smooth surface of type $^{II} Z_{E}^{10}$. Then the linear system of cubic complexes passing through $S$ induces a special birational map $\varphi$ from $Q^4$ onto a quadric hypersurface $Z$ in $\PP ^5$ of rank $\geq 3$ whose base locus scheme is the surface $S$. In particular, $Bl_S(Q^4)$ is a Fano fourfold of index $1$ anti-canonically embedded in $Q^4\times Z\subseteq \PP^5 \times \PP^5 $.
\end{pr}
\begin{proof}
    The spinor bundles $\mathcal{E}$ and $\mathcal{E}'$ and the ideal sheaf $\mathcal{I}_{S/Q^ 4}$ fit in the following short exact sequence by Proposition \ref{resolution of ideal sheaf of S types E and B}: 
    \[0\rightarrow \mathcal{E}(-3)\oplus \mathcal{E}'(-3)\oplus \ko_{Q^4}(-4)\rightarrow \ko_{Q^4}(-3)^{\oplus 6} \rightarrow \mathcal{I}_{S/Q^ 4}\rightarrow 0.\]
    An easy computation shows $H^0(Q^4 , \mathcal{I}_{S/Q^ 4} (3))= H^0(Q^4 , \ko_{Q^4} ^{\oplus 6})= \CC ^6$. Hence, the linear system $|H^0 (Q^4, \mathcal{I}_{S/Q^ 4} (3))|\subseteq |H^0(Q^4, \ko_{Q^4} (3))|$ defines a rational map $\varphi: Q^4 \dashrightarrow \PP^5 $. Moreover, since the map $\ko_{Q^4}(-3)^{\oplus 6} \rightarrow\mathcal{I}_{S/Q^4 }$ is surjective, the six cubic polynomial equations defining a basis of $H^0(Q^4 , \mathcal{I}_{S/Q^ 4} (3))$ also generate a homogeneous ideal in the coordinate ring of $Q^4$ with associated ideal sheaf $\mathcal{I}_{S/Q^ 4}$, or equivalently, closed subscheme $S$. In other words, $S$ is the scheme-theoretic intersection of the basis elements of the linear system defining $\varphi$. This shows that $\varphi$ is special with base locus scheme $S$.

    Let $X:=Bl_S (Q^4)$ and $E$ be the exceptional divisor of the blow-up $\sigma : X\rightarrow Q^4$. Then the linear system $|3H-E|$ defines a morphism $\tau: X\rightarrow \PP ^5$ extending $\varphi$ via $\sigma $. If we substitute the numerical data of $S$ into the equations of Lemma \ref{intersection theory} (case $k=4$), we get $(3H-E)^4 = 2$. Hence, $\varphi$ is birational and the image $Z:=\tau(X)$ is a four-dimensional non-degenerate variety of degree $2$ in $\PP^5$, i.e., a quadric hypersurface of rank at least $3$.

    Finally, since the induced birational map $\varphi: Q^4\dashrightarrow Z$ is special, the blow-up $Bl_S(Q^4)$ can be identified with the graph $\Gamma _\varphi \subseteq Q^4\times Z$ of $\varphi$. In particular, $\ko_{Q^4\times Z}(1,1)$ restricts to the ample divisor $H+H'$ on $Bl_S(Q^4)$. But the anti-canonical divisor of $Bl_S(Q^4)$ can be rewritten as $-K_{Bl_S(Q^4)} =4H-E=H+(3H-E)=H+H'$, so it is ample. Moreover, using the decomposition $Pic( Bl_S(Q^4))=\ZZ H \oplus \ZZ E$ we can easily check that $H+H'$ is a primitive divisor. Hence, the index is $1$.
\end{proof}

A quadric of rank $\geq 3$ in $\PP^5$ is normal by \cite[Ch. II, Proposition~8.23 (b)]{hartshorne}. However, if we want to get a true converse of Proposition \ref{characterization of special bir trans of 4-quadric} we need to show that $Z$ is actually smooth. As a first step we introduce a variety that will help us understand the geometry of the map $\varphi: Q^4 \dashrightarrow Z$. 

\begin{defn} Let $\mathcal{V}$ be a non-degenerate subvariety of $\PP^k $ and $\ell $ a line not contained in $\mathcal{V}$. We say that $\ell$ is a trisecant of $\mathcal{V}$ if the scheme-theoretic intersection $\ell\cap \mathcal{V} $ is zero-dimensional of length $\geq 3$. The trisecant variety of $\mathcal{V}$ (in $\PP^k$) is defined as $\text{Trisec}(\mathcal{V}):=\bigcup \{ \ell\text{ is a line in }\PP^k : \ell\subseteq \mathcal{V} \text{ or } \ell \text{ is a trisecant of } \mathcal{V} \}$.
\end{defn}

Observe that $\text{Trisec}(S) \subseteq  Q^4$ since $S\subseteq Q^4$ and any line intersecting a quadric hypersurface in more than two points (counted with multiplicity) is actually contained in it. 

\begin{pr}\label{dim Trisec >= 3} We have $dim(\text{Trisec}(S)) \geq 3$ and $S\subseteq \text{Trisec}(S)\subseteq Q^4$. In particular, $S$ has a trisecant line.
\end{pr}
\begin{proof}
    In \cite[Theorem~1]{bauer98} I. Bauer provides the classification of smooth, non-degenerate, connected complex surfaces $\Sigma$ in $\PP^5$ which are not scrolls and such that $dim(\text{Trisec}(\Sigma))\leq 2$. In fact, the last condition is also equivalent to $\text{Trisec}(\Sigma)\cap \Sigma\neq \Sigma$ by \cite[Corollary~1]{bauer98}. There are just eight cases in her classification, none of which can be our surface $S$. Thus, the trisecant variety of $S$ must have dimension at least $3$ and contains the surface $S$.
\end{proof}

\begin{pr}\label{true trisecants}
    All trisecants of $S$ necessarily intersect $S$ in a zero-dimensional scheme of length exactly $3$. 
\end{pr}
\begin{proof}
    Let $c_1,\dots, c_6$ be cubic equations giving a basis of $H^0(Q^4 , \mathcal{I}_{S/Q^ 4} (3))$ and let $q$ be the equation of the quadric $Q^4$. Also, denote by $C_i$ the cubic in $\PP^5$ determined by $c_i$. Then the homogeneous ideal $\langle q, c_1, c_2, \dots ,c_6 \rangle$ in $\CC[z_0, \dots, z_5]$ cuts the surface $S$ in $\PP^5$ scheme-theoretically, i.e., $S=\bigcap_{i=1}^{6} C_i \cap Q^4$. Now let $\ell$ be a trisecant line of $S$. If $J$ is the homogeneous ideal of $\ell$ in $\CC[z_0, \dots, z_5]$, then $\frac{\CC[z_0, \dots, z_5]}{J + \langle c_i \rangle}\rightarrow \frac{\CC[z_0, \dots, z_5]}{J + \langle q, c_1, c_2, \dots ,c_6\rangle}$ is surjective. Hence, the induced map on local rings $\ko_{p,\ell\cap C_i} \rightarrow \ko_{p,\ell\cap S}$ is surjective for all $p\in supp(\ell\cap S)$ and all $i$. In particular, $\text{length}(\ell\cap S )\leq \text{length}(\ell\cap C_i)$ if $\ell\nsubseteq C_i$. Of course, there is at least one $i$ such that $\ell\nsubseteq C_i$. Otherwise, we would have $\ell\subseteq \bigcap_{i=1}^{6} C_i \cap Q^4=S$, which is not possible by the definition of trisecant line. Thus, for such $i$ we have $3\leq \text{length}(\ell\cap S )\leq \text{length}(\ell\cap C_i)\leq 3$, i.e., $\text{length}(\ell\cap S )=3$.
\end{proof}

We study now the structure of the birational morphism $\tau$ with elementary techniques from the minimal model program. A reference for this subject is \cite{Debarre_Higuer}.

\begin{pr}\label{tau is contraction of extremal ray}
  The class of the strict transform of any trisecant line of $S$ induces an extremal ray $R$ of the Mori cone $\overline{NE}(X)$. The ray $R$ is not nef and $H'$ is a good supporting divisor for $R$. In fact, $\tau$ is the contraction morphism induced by $R$. 
\end{pr}
\begin{proof}
    For $Q^4$ the numerical and rational equivalence relations on algebraic cycles coincide and we have
    \[CH^{\ast}(Q^4)=\frac{\ZZ [H,P]}{\langle H^3-2HP, P^2-H^2 P\rangle}.\]
    Here $H$ is the hyperplane class and $P$ is a plane from one of the two rulings of the quadric $Q^4$. The grading is $1$ for $H$, and $2$ for $P$. Since $HP$ can be represented by the class of a line $l$ in $Q^4$, we have $N_1(Q^4)=CH^3(Q^4)=\frac{\langle H^3,HP\rangle}{\langle H^3-2HP\rangle}=\ZZ l$.

    Let $f$  be a fiber in the exceptional divisor $E$ of the blow-up $\sigma: X\rightarrow Q^4$. We know that $\sigma $ is the contraction morphism induced by the extremal ray spanned by $f$ in $\overline{NE}(X)$. Hence, if we denote the pullback of $l$ under $\sigma$ by $l$ again, we have $N_1(X)_{\RR}=\RR l\oplus \RR f$. 

    If $\ell$ is a trisecant line of $S$, then the numerical class of its strict transform $\tilde{\ell}$ under $\sigma$ is $l-3f$ by Proposition \ref{true trisecants}. We define $R$ to be the ray $\RR ^+ [\tilde{\ell}]\subseteq \overline{NE}(X)$. Observe that $R$ is $K_X $-negative since $X$ is Fano, or simply because $K_X(l-3f) =(-4H+E)(l-3f)=-1$. Also, the ray $R$ is contracted by $\tau$ since $H'(l-3f)=(3H-E)(l-3f)=0$. Moreover, if a curve in $X$ has class $al+bf$, $a,b\in \RR$, $a\geq 0$, and is contracted by $\tau$, then $H'(al+bf)=(3H-E)(al+bf)=3a+b=0$, implying that its numerical class is $a(l-3f)\in R$. In particular, $R$ is an extremal ray and $\tau$ is the contraction morphism induced by $R$. The statements about $R$ and $H'$ are clear.
\end{proof}

\begin{cor}\label{trisec=Exc(R)}
The morphism $\tau$ is a divisorial contraction induced by the extremal ray $R$ and its exceptional locus $Exc(R)$ is the strict transform $\widetilde{\text{Trisec}(S)}$ of the trisecant variety ${\text{Trisec}(S)}$ of $S$ under $\sigma$. 
\end{cor}
\begin{proof}
    Evidently $\widetilde{\text{Trisec}(S)}$ is contained in the exceptional locus $Exc(R)$ of the extremal ray $R$. In particular, $3\leq dim(\widetilde{\text{Trisec}(S)})\leq dim(Exc(R))\leq 3$ by Proposition \ref{dim Trisec >= 3} and by the fact that $\tau$ is birational. Hence, $Exc(R)=\widetilde{\text{Trisec}(S)}$ is a prime divisor of $X$. 
\end{proof}

We denote $T:=\tau(Exc(R))\subseteq Z$.

\begin{pr}\label{T is a surface}
    $T$ is irreducible of dimension $2$.
\end{pr}

\begin{proof}
    The dualizing sheaf of $Z$ is the line bundle $\ko_Z(-4)$, so the canonical divisor $K_Z$ is Cartier. Therefore, we can write $K_X=\tau ^{\ast} K_Z +\lambda Exc(R)$ in $Pic(X)$, $\lambda\in \ZZ$, for $\tau$ is an isomorphism outside $Exc(R)$. In fact, we have $\lambda=1$, since $K_X R=-1<0$, $Exc(R)R<0$, and $\tau^{\ast}Pic(Z)=ker(\times R:Pic(X)\rightarrow \ZZ)$. Thus, $-4H+E=K_X=\tau ^{\ast} K_Z +Exc(R)=-4H'+Exc(R)$, which gives $Exc(R)=8H-3E$ in $Pic(X)$. Now we compute $H'^2{Exc(R)}^2=(3H-E)^2(8H-3E)^2=-10\neq 0$ using Lemma \ref{intersection theory}. This implies that $T$ has dimension $2$ by \cite[Proposition~2.7]{beltrametti.notnef}. Finally, $T$ is irreducible since $Exc(R)$ is.
\end{proof}

If the morphism $Exc(R)\rightarrow T$ is equidimensional (all fibers are of dimension $1$), then Proposition \ref{T is a surface} implies that both $Z$ and $T$ are smooth by the well-known result of T. Ando \cite[Theorem~2.3]{ando}. Nevertheless, it may very well happen that fiber dimension is non-constant; examples of this phenomenon are abundant. In any case, we know that if there exists a fiber of dimension $2$ then it is necessarily isolated by Proposition \ref{T is a surface}. In fact, it is certain that two-dimensional fibers exist over singular points of $Z$ or $T$ by \cite[Theorem~4.1]{AW_AView}. Of course, then the singularities of $Z$ and $T$ are isolated. This implies that $Z$ can have a nodal singularity at worst. 

Fortunately, a classification of isolated two-dimensional fibers in Fano-Mori divisorial contractions of fourfolds is given by M. Andreatta and J. A. Wi\'sniewski in \cite[p.~3, Theorem]{AWI}. According to their classification, a special fiber $F=\tau ^{-1}(z)$, $z\in Z$, is either isomorphic to a plane or to a quadric in $\PP^3$ of rank at least $2$. If $F$ is isomorphic to a plane, the conormal bundle $N^{\ast}_{F/X}$ can be identified with either $\frac{T(-1)\oplus \ko(1) }{\ko}$ or $\frac{\ko^{\oplus 4}}{\ko(-1)^{\oplus 2}}$. Moreover, in the first case $z$ is a node of $Z$, but a smooth point of $T$, while in the second case $z$ is a smooth point of $Z$, but a singular point of $T$ isomorphic to the vertex of a cone over a rational twisted cubic curve. If $F$ is isomorphic to a quadric, the conormal bundle $N^{\ast}_{F/X}$ can be identified with the restriction of $\mathcal{E}(1)=\mathcal{E}^{\ast}$, where $\mathcal{E}$ is the spinor bundle of $Q^{4}$, and $z$ is a smooth point of $Z$, but a non-normal singular point of $T$. 

Therefore, in order to show that $Z$ is smooth it suffices to discard the existence of plane fibers in $Exc(R)\rightarrow T$. We give a lattice-theoretic condition on the Picard group of the K3 surface $S_0$ that prevents this situation. Observe, however, that $Z$ could be smooth even if there exist plane fibers, although $T$ would necessarily be singular then.

\begin{pr} \label{geometric constraints on S} Let $S_0$ be the minimal model of $S$ in $\PP^7$ and let $H_{S_0}$ be the hyperplane class of $S_0$. If $Exc(R)\rightarrow T$ has a plane fiber, then there exists a smooth elliptic curve $\mathfrak{E}\subseteq S_0$ of degree $5$ such that the sublattice $\left\langle H_{S_0}, \mathfrak{E}\right\rangle\subseteq Pic(S_0)$ is saturated and the points to which $d$ and $d'$ are contracted are contained in a single (possibly singular) fiber of the associated elliptic fibration of $S_0$. More generally, if $Exc(R)\rightarrow T$ is not equidimensional, then $rk(Pic(S_0))>1$.
\end{pr}
\begin{rem}\label{remark}
    Using \texttt{Macaulay2} it is possible to construct an explicit example of a K3 surface of degree $12$ in $\PP^7$ containing a smooth elliptic curve of degree $5$ such that its Picard group is generated by the hyperplane class and the class of the elliptic curve. Moreover, a random projection from two of its points induces a surface of type $^{II} Z_{E}^{10}$ in $\PP^5$.
\end{rem}
\begin{proof}
    Assume that there exists a two-dimensional fiber $F=\tau ^{-1}(z)$, $z\in T\subseteq Z$. Let $F_1$ be any irreducible component of $F$. Observe that the restricted morphism $\sigma_{|_{F_1}}: F_1\rightarrow \sigma(F_1)$ is finite, since $\{0\}=\RR ^{+} [l-3f] \cap \RR^{+}[f]\subseteq \overline{NE}(X)$. Hence, $F_1\nsubseteq E$. Indeed, otherwise we would have a surjective morphism $\sigma_{|_{F_1}}: F_1\rightarrow S$ exhibiting $S_0$ as an unirational surface. This is not the case. Therefore, the morphism $\sigma_{|_{F_1}} : F_1\rightarrow \sigma(F_1)$ is birational.

\textit{Case $1$}. $F\cong \PP^2$.

    Let us prove that $\sigma(F)\subseteq Q^4$ is a plane that intersects $S$ in a cubic curve. Indeed, we compute $\ko_X (4H-E)_{|F}={\ko_X (-K_X)}_{|F}=\ko_F (-K_F )\otimes det(N_{F/X})=\ko_{\PP^2}(3)\otimes \ko_{\PP^2}(-2)=\ko_{\PP^2}(1)$ and $\ko_X (3H-E)_ {|F}=0$, since $F$ is contracted to a point under $ \tau=|3H-E|$. Hence, we have $\ko_X (H)_{|F}=\ko_{\PP^2}(1)$ and $\ko_X (E)_{|F}=\ko_{\PP^2}(3)$. This implies that $\sigma (F)$ has degree $1$, i.e., it is a plane. Thus, $\sigma_{F}$ is an isomorphism. In particular, the intersection $\mathcal{C}:=S\cap \sigma(F)\cong E\cap F$ is a cubic in the plane $\sigma(F)$. 

    Now we compute the degree of the image of the cubic $\mathcal{C}$ under the blow-up map $b_S: S\rightarrow S_0\subseteq \PP^7$ induced by the complete linear system $|H_S+K_S |$ on $S$. First observe that the Segre class of $\mathcal{C}$ in $\sigma(F)$ is $s(\mathcal{C},\sigma(F))=1-3H_{\mathcal{C}}t$, where we denote by $H_{\mathcal{C}}$ the restriction of $H$ to $\mathcal{C}$. On the other hand, the total Chern class of the normal bundle of $S$ in $Q^4$ is $c(N_{S/Q^4})=1+ (K_S +4H_S) t + 50 t^2$ by Proposition \ref{total chern class of normal bundle of S}. Let $(S\cdot \sigma(F))^{\mathcal{C}}$ be the equivalence of $\mathcal{C}$ for the intersection $S\cdot \sigma(F)$ (see \cite[Ch. 9]{Fulton} for the definition of equivalence). Since $\mathcal{C}$ is connected, we have $S\cdot \sigma(F)=(S\cdot \sigma(F))^{\mathcal{C}}
    =\left\{c ({N_{S/Q^4}}_{|\mathcal{C}})\frown s(\mathcal{C},\sigma(F))\right\}_0$ by \cite[Proposition~9.1.1]{Fulton} and \cite[Proposition~9.1.2]{Fulton}. Clearly,
    \begin{align*}
        c ({N_{S/Q^4}}_{|\mathcal{C}})\frown  (1-3H_{\mathcal{C}}t) &=c ({N_{S/Q^4}}_{|\mathcal{C}})\frown(1-3H_{S}t)_{|\mathcal{C}}\\
        &=[c ({N_{S/Q^4}})\frown (1-3H_{S}t)]_{|\mathcal{C}}\\
        &=[(1+ (K_S +4H_S) t + 50 t^2)(1-3H_{S}t)]_{|\mathcal{C}}\\
        &=1+(H_{S}+{K_{S}})_{|\mathcal{C}}t.
    \end{align*}
    Therefore, 
    \begin{align*}
        5=(5P_1+5P_2)P_i &=S\cdot \sigma(F)\\
        &=\left\{c ({N_{S/Q^4}}_{|\mathcal{C}})\frown s(\mathcal{C},\sigma(F))\right\}_0\\
        &=(H_{S}+{K_{S}})\mathcal{C}.
    \end{align*}
    This implies that the degree of $b_S(\mathcal{C})\subseteq S_0$ is $5$. In particular, $K_{S}\mathcal{C}=2$, and thus $\mathcal{C}^2=-2$ by the adjunction formula, for plane cubics have arithmetic genus equal to $1$. 
    
    Write $\mathcal{C}=b_S^{\ast}(\mathcal{D})+x d+y d^{\prime}$ in $Pic(S)$, where $\mathcal{D}\in Pic(S_0)$ and $x, y\in \ZZ$. Thus, $\mathcal{D}={b_S}_{\ast}(\mathcal{C})=b_S(\mathcal{C})$ in $Pic(S_0)$. Then $2=K_S \mathcal{C}=\left(d+d^{\prime}\right)\left(b_S^{\ast}(\mathcal{D})+x d+y d'\right)=$ $-x-y$ and $-2=\mathcal{C}^2=\left(b_S^{\ast}(\mathcal{D})+x d+y d'\right)^2=\mathcal{D}^2-x^2-y^2$. We want to find the possible pairs $(x,y)$ and compute $\mathcal{D}^2$. To begin with, notice that $d$ and $d^{\prime}$ cannot be components of the plane cubic $\mathcal{C}$ simultaneously, for they are disjoint. If neither $d$ nor $d^{\prime}$ are components of $\mathcal{C}$, then $d \mathcal{C}=-x$ and $d^{\prime} \mathcal{C}=-y$ are non-negative, implying $(x, y)=(- 1,-1)$, $(-2,0)$, or $(0,-2)$. However, if one of these lines is a component of $\mathcal{C} $, for example $d$, then $d'$ is not contained in the plane $\sigma(F)$. Clearly, $d^{\prime} \mathcal{C}=-y=0$, $1$, or $2$, which also gives $(x, y)=(- 1,-1)$, $(-2,0)$, or $(0,-2)$. We thus have $\mathcal{D}^2= 0$ or $2$. Assume $\mathcal{D}^2=2$. Since $H_{S_0}\mathcal{D}=5$, we have $\left(H_{S_0}-2 \mathcal{D}\right)^2=0$ and $H_{S_0}\left(H_{S_0}-2 \mathcal{D}\right)=2$. Then the restriction of $\ko_{S_0}\left(H_{S_0}-2 \mathcal{D}\right)$ to each smooth curve in $|H_{S_0}|$ defines a $g_{2}^1$. This is absurd, since the inclusion $S_0 \subseteq \mathbb{P}^7$ restricts to the canonical map on each smooth curve in $|H_{S_0}|$. Therefore, $\mathcal{D}^2=0$ and $d\mathcal{C}=d'\mathcal{C}=1$. 
    
    We claim that $\mathcal{D}$ is nef. To prove this, it suffices to verify that the class of any irreducible component of $b_S(\mathcal{C})$ has non-negative intersection with $\mathcal{D}$. This is certainly the case if $b_S(\mathcal{C})$ is a curve, for $\mathcal{D}^2=0$. Assume that $\mathcal{C}$ is reducible, but it contains no $(-1)$-line. Then the decomposition of $\mathcal{C}$ as an effective Cartier divisor in the plane $\sigma(F)$ is either $3 \mathfrak{l}_1$, $2 \mathfrak{l}_1+\mathfrak{l}_2$, $\mathfrak{l}_1+\mathfrak{l}_2 +\mathfrak{l}_3$, or $\mathfrak{l}_1+\mathfrak{c}$, where $\mathfrak{l}_1$, $\mathfrak{l}_2$, and $\mathfrak{l}_3$ are three pairwise distinct lines and $\mathfrak{c}$ is a smooth conic. If $\mathcal{C}=3 \mathfrak{l}_1$, then $-2=\mathcal{C}^2=9\mathfrak{l}_1^2$, which is not possible. If $\mathcal{C}=2 \mathfrak{l}_1+\mathfrak{l}_2$, then $d$ and $d'$ each intersect $\mathfrak{l}_2$ in a point away from $\mathfrak{l}_1$, so $\mathfrak{l}_2^2=-4$ by the adjunction formula. Hence, $-2=\mathcal{C}^2=(2\mathfrak{l}_1 +\mathfrak{l}_2)^2=4\mathfrak{l}_{1}^2$, which is also not possible. Therefore, $\mathcal{C}=\mathfrak{l}_1+\mathfrak{l}_2 +\mathfrak{l}_3$ or $\mathfrak{l}_1+\mathfrak{c}$. Observe that $d$ and $d'$ each intersect $\mathcal{C}$ transversely at a point away from the intersection points of its components. If $\mathcal{C}=\mathfrak{l}_1+\mathfrak{l}_2 +\mathfrak{l}_3$, then $\mathcal{D}={b_S}_{\ast}(\mathcal{C})=\mathfrak{r}_1+\mathfrak{r}_2+\mathfrak{r}_3$ in $Pic(S_0)$, where $\mathfrak{r}_1$, $\mathfrak{r}_2$, and $\mathfrak{r}_3$ are three smooth rational curves pair-wisely intersecting transversely at a point. This implies $\mathfrak{r}_i\mathcal{D}=0$. If $\mathcal{C}=\mathfrak{l}_1+\mathfrak{c}$, then $\mathcal{D}={b_S}_{\ast}(\mathcal{C})=\mathfrak{r}_1+\mathfrak{r}_2$ in $Pic(S_0)$, where $\mathfrak{r}_1$ and $\mathfrak{r}_2$ are two smooth rational curves such that $\mathfrak{r}_1\mathfrak{r}_2=2$. This implies $\mathfrak{r}_i\mathcal{D}=0$. Now assume that $\mathcal{C}$ is reducible and contains a $(-1)$-line, for example $d$. Then the decomposition of $\mathcal{C}$ as an effective Cartier divisor in the plane $\sigma(F)$ is either $3 d$, $2 d+\mathfrak{l}_1$, $ d+2 \mathfrak{l}_1$, $d+\mathfrak{l}_1+\mathfrak{l}_2$, or $d+\mathfrak{c}$, where $\mathfrak{l}_1$ and $\mathfrak{l}_2$ are distinct lines that are also different from $d$, and $\mathfrak{c}$ is a smooth conic. By intersecting with $d$ or $d'$ we can discard the first three cases. Therefore, $\mathcal{C}=d+\mathfrak{l}_1+\mathfrak{l}_2$ or $d+\mathfrak{c}$. Observe that $d'$ intersects $\mathcal{C}$ transversely at a point away from the intersection points of its components and away from $d$. If $\mathcal{C}=d+\mathfrak{l}_1+\mathfrak{l}_2$, then $\mathcal{D}={b_S}_{\ast}(\mathcal{C})=\mathfrak{r}_1+\mathfrak{r}_2$ in $Pic(S_0)$, where $\mathfrak{r}_1$ and $\mathfrak{r}_2$ are two smooth rational curves intersecting transversely at two points if $d$, $\mathfrak{l}_1$, and $\mathfrak{l}_2$ are not concurrent, or they are tangent at a single point otherwise. Thus, $\mathfrak{r}_1\mathfrak{r}_2=2$. This implies $\mathfrak{r}_i\mathcal{D}=0$. If $\mathcal{C}=d+\mathfrak{c}$, then $b_S(\mathcal{C})$ is a rational curve with an ordinary double point if $d$ intersects transversely $\mathfrak{c}$, or it is a rational curve with a cusp if $d$ is tangent to $\mathfrak{c}$, and thus there is nothing to do. The claim follows from the previous analysis by cases.
     
    We now claim that $\mathcal{D}$ is primitive in $Pic(S_0)$. Assume on the contrary that $\mathcal{D}=s\mathcal{D}'$ for some $\mathcal{D}'\in Pic(S_0)$ and $s>0$. Then $sH_{S_0}\mathcal{D}'=5$, so we have $s=1$ or $5$. If $s=5$, then we have $H_{S_0}\mathcal{D}'=1$ and $\mathcal{D}'^2=0$. This implies that $H_{S_0}$ is not base-point free, since $H_{S_0}$ is ample and $\mathcal{D}'$ is effective. This is absurd. Consequently $s=1$ and the claim follows. Notice that $\mathcal{D}$ being primitive in $Pic(S_0)$ is equivalent to $\left\langle \mathcal{D}\right\rangle$ being saturated in $Pic(S_0)$.

    We conclude that $\mathcal{D}$ is isotropic, nef and primitive. Hence, there exists a smooth elliptic curve $\mathfrak{E}\subseteq S_0$ such that $\mathfrak{E}=\mathcal{D}$ in $Pic(S_0)$ by \cite[Ch. 2, Proposition~3.10]{Huybrechts_K3}. In particular, $b_S(\mathcal{C})$ is a (possibly singular) fiber of the associated elliptic fibration of $S_0$ and $d$ and $d'$ are contracted to points in it.
    
    Lastly, let us show that $\left\langle H_{S_0}, \mathfrak{E}\right\rangle$ is a saturated sublattice of $Pic(S_0)$. Indeed, let $M$ be the saturation of $\left\langle H_{S_0}, \mathfrak{E}\right\rangle$ in $Pic(S_0)$. Then $5^2=disc(\left\langle H_{S_0},\mathfrak{E}\right\rangle)=|disc(M)|[M:\left\langle H_{S_0}, \mathfrak{E}\right\rangle]^2$. Hence, $[M:\left\langle H_{S_0}, \mathfrak{E}\right\rangle]=1$ or $5$. Assume that the index is $5$. Then there exists $ \zeta \in M$ such that $\zeta\notin \left\langle H_{S_0}, \mathfrak{E}\right\rangle$, but $5\zeta \in \left\langle H_{S_0}, \mathfrak{E}\right\rangle$. Write $5\zeta =\alpha H_{S_0}+ \beta \mathfrak{E}$, where $\alpha,\beta\in \ZZ$. Then $25\zeta^2= 12\alpha ^2 +10\alpha\beta$, so we have $5\vert \alpha$. Hence, $\eta:=\zeta-\frac{\alpha}{5}H_{S_0}\in M$ satisfies $5\eta =\beta \mathfrak{E}$. Then $\eta\in \left\langle \mathfrak{E}\right\rangle$, implying $\zeta\in \left\langle H_{S_0}, \mathfrak{E}\right\rangle$. This is a contradiction. Therefore $[M:\left\langle H_{S_0}, \mathfrak{E}\right\rangle]=1$ and the claim follows.
   
\textit{Case $2$}. $F\cong V_2 \subseteq \PP^3$, where $V_2$ is a quadric of rank at least $2$.

    Let us prove that $\sigma(F)\subseteq Q^4$ is a quadric of the same type as $F$ that intersects $S$ in a sextic. Indeed, we compute $\ko_X (4H-E)_{|F}=\ko_ {V_2} (1) $ using the adjunction formula and the identification $det(\mathcal{E}^{\ast})=\ko_{Q^4}(1)$. As before, this implies $\ko_X (H)_{|F}=\ko_{V_2}(1)$ and $\ko_X (E)_{|F}=\ko_{V_2}(3)$. 

    \textit{Subcase $1$}. $F=F_1 \cup F_2$ is isomorphic to the union of two planes meeting along a line.
    
        By restricting $\ko_X (H)$ and $\ko_X (E)$ to $F_i$, we have $\ko_X (H)_{|F_i}=\ko_{\PP^ 2}(1)$ and $\ko_X (E)_{|F_i}=\ko_{\PP^2}(3)$. In particular, $\sigma(F_1)$ and $\sigma(F_2)$ have degree $1$, i.e., they are planes too. Hence, $\sigma_{|_{F_1}}$ and $\sigma_{|_{F_2}}$ are isomorphisms, and thus the intersections $\mathcal{C}_1:=S\cap \sigma(F_1)\cong E\cap F$ and $\mathcal{C}_2:=S\cap \sigma(F_2)\cong E\cap F$ are cubics. Moreover, $\sigma(F_1)$ and $\sigma(F_2)$ are distinct, since $\sigma: X-E\rightarrow Q^{4}-S$ is an isomorphism and $F_1 ,F_2\nsubseteq E$. Thus, they intersect in a proper linear subspace, which must be a line, since $F_1 \cap F_2\cong \PP^1$. In summary, $\sigma: F\rightarrow \sigma(F)$ is an isomorphism, and thus we can apply the argument of \textit{Case 1} to show $rkPic(S_0)\geq 3>1$.
    
    \textit{Subcase $2$}. $F$ is isomorphic to a smooth quadric or a quadric cone.
    
        The computation above implies that $\sigma(F)$ has degree $2$, i.e., it is a quadric. Moreover, the finite birational map $\sigma: F\rightarrow\sigma(F)$ is an isomorphism, since $\sigma(F)$ is necessarily normal due to irreducibility. Hence, the intersection $\mathcal{C}:=S\cap \sigma(F)\cong E\cap F$ is a sextic in $\sigma(F)$.
    
        Let us compute the intersection number $K_S \mathcal{C}$. Since the sextic $\mathcal{C}$ is a Cartier divisor in $\sigma(F)$ by construction, the Segre class of $\mathcal{C}$ in $\sigma(F)$ is $s(\mathcal{C},\sigma(F))=1-3H_{\mathcal{C}}t$. Moreover, $\mathcal{C}$ is also connected by \cite[Ch. III, Corollary~7.9]{hartshorne}, so we have 
        \begin{align*}
            10=(5P_1+5P_2)(P_1+P_2)&=S\cdot \sigma(F) \\
            &=(S\cdot \sigma(F))^\mathcal{C}\\
            &=\left\{c ({N_{S/Q^4}}_{|\mathcal{C}})\frown s(\mathcal{C},\sigma(F))\right\}_0\\
            &=(H_S+{K_{S}})\mathcal{C}.
        \end{align*}
        This implies ${K_{S}}\mathcal{C}=4$. 
        
        Assume $rk(Pic(S_0))=1$ and write $\mathcal{C}=uH_{S}+vd+wd'$ in $Pic(S)$, where $u,v,w\in \ZZ$. Then we have the following system of equations
        \[\left\{ 
        \begin{array}{l}
        H_{S}\mathcal{C}=H_{S}(uH_{S}+vd+wd')=10u+v+w=6 \\
        {K_{S}}\mathcal{C}=(d+d')(uH_{S}+vd+wd')=2u-v-w=4.
        \end{array} 
        \right.\]
        
        It is easy to see that no integer solutions exist for this system. This is a contradiction.   
\end{proof}

A quadric in $\PP^5$ with a nodal singularity at worst is factorial by Grothendieck's parafactoriality theorem (see \cite[Chap. XI, Corollaire~3.14]{SGA2}). Hence, the base locus scheme $\mathfrak{B}'$ of the inverse $\varphi^{-1}:Z\dashrightarrow Q^4$ is defined and $\tau$ is the blow-up of $Z$ along $\mathfrak{B}'$. Its exceptional divisor $E'$ is reduced by the same proof of Proposition \ref{equations in PicX}. By Zariski's main theorem, we then have $\mathfrak{B}'_{red}=\{z\in Z : dim(\tau^{-1}(z))\geq 1\}$ and all fibers are connected. Therefore, $E'=Exc(R)$ and $\mathfrak{B}'_{red}=T$.

\begin{proof}[(Proof of Theorem \ref{main thm})]
    The first claim is just Proposition \ref{characterization of special bir trans of 4-quadric}. The converse claim follows from Proposition \ref{S induces a special birational transformation onto a quadric} and Proposition \ref{geometric constraints on S}. Assume that $rkPic(S_0)=1$. Then $T$ is smooth, and thus $\mathfrak{B}'=T$ by (\textit{a}) in Proposition \ref{equations in PicX}. Let $T_0$ be the minimal model of $T$. Clearly, the K3 surfaces $S_0$ and $T_0$ are derived-equivalent by \cite[Lemma~3.4]{Fano4foldsK3type}. Moreover, if $H_{S_0}$ is the induced polarization of degree $12$ on $S_0$, then the Fourier-Mukai partners of $S_0$ are isomorphic to either $S_0$ or the moduli space of stable sheaves over $S_0 $ with Mukai vector $(2,H_{S_0},3)\in \widetilde{H}(S_0 ,\ZZ)$ by \cite[Proposition~1.10]{oguiso.almost-primes}. However, we can show $S_0\ncong T_0$ in the same way we proved Theorem \ref{T0 is not S0}.
\end{proof}

\end{document}